\documentclass[11pt]{article}
\usepackage[top=2.5cm, bottom=2.5cm, left=2.5cm, right=2.5cm] {geometry}
\usepackage{lineno}
\usepackage{comment}
\usepackage{amsmath,amssymb,theorem,color,amsfonts,mathrsfs}
\numberwithin{equation}{section} 
\usepackage{amssymb}
\usepackage{color}

\usepackage{amsmath}
\usepackage{graphicx}
\usepackage{amsfonts}%
\newcommand{\R}{\ensuremath{\mathbb{R}}}

\newcommand{\N}{\ensuremath{\mathbb{N}}}

\newcommand{\cA}{\mathcal{A}}
\newcommand{\cC}{\mathscr{C}}

\newcommand{\bP}{\mathbb{P}}
\newcommand{\cP}{\mathcal{P}}

\newcommand{\mP}{\mathbb{P}}

\newcommand{\X}{\mathbb{X}}

\newcommand{\E}{\mathbb{E}}

\newcommand{\cD}{\mathcal{D}}

\newcommand{\cL}{\mathcal{L}}  

\newcommand{\bx}{\mathbf x}
\newcommand{\bX}{\mathbf X}

\newcommand{\cB}{\mathcal{B}}

\newcommand{\tp}{p{\rm -var}}
\newcommand{\tq}{q{\rm -var}}

{\Large }

\newcommand{\ltn}{\ensuremath{\left| \! \left| \! \left|}}
\newcommand{\rtn}{\ensuremath{\right| \! \right| \! \right|}}

\newtheorem{theorem}{Theorem}
{ \theorembodyfont{\normalfont} 
	\newtheorem{example}[theorem]{Example}
	\newtheorem{remark}[theorem]{Remark}
}

\newtheorem{lemma}[theorem]{Lemma}
\newtheorem{corollary}[theorem]{Corollary}
\newtheorem{proposition}[theorem]{Proposition}

\newcounter{enumctr}

\makeatletter
\def\and{%
\end{tabular}%
\begin{tabular}[t]{c}}%
\def\@fnsymbol#1{\ensuremath{\ifcase#1\or a\or b\or c\or
		d\or e\or f\or g\or h\or i\else\@ctrerr\fi}}
\makeatother

\begin{document}
	
	\title{Taming discrete rough paths via strong Lyapunov functions}
	\author{Luu Hoang Duc\thanks{Department of Mathematics, University of Klagenfurt, Austria \& Max Planck Institute for Mathematics in the Sciences, Leipzig, Germany \& Institute of Mathematics, Vietnam Academy of Science and Technology, Vietnam. {\tt \small duc.luu@aau.at, duc.luu@mis.mpg.de, lhduc@math.ac.vn}} }
	\date{\bf working paper}
	\maketitle
	
	\begin{abstract}
     Based on the newly introduced concept of {\it strong Lyapunov functions} for rough differential equations \cite{ducjost25}, we study a tamed numerical scheme to approximate the solutions of the continuous system. We derive explicit estimates of solution norms of the tamed system which look similar to those of the continuous system. As a result, we prove the convergence of the tamed scheme in the $L^1$ sense. For systems with the negative gradient condition, we prove the existence of a numerical pullback attractor for the generated random dynamical system from the tamed numerical scheme which is integrable and upper semi-continuous w.r.t. the scheme step size. 
	\end{abstract}
	
	{\bf Keywords:}
	rough differential equations, tamed numerical scheme, Doss-Sussmann transformation, strong Lyapunov functions, pullback attractors, upper semi-continuity.
	
	\section{Introduction}\label{sec:intro}
  Consider the rough differential equation
\begin{equation}\label{fSDE0}
	dy_t = f(y_t)dt + g(y_t)d x_t.
\end{equation}
The solution of \eqref{fSDE0} is often understood in the sense of either Lyons-Davie \cite{lyons98} or of Friz-Victoir \cite{friz}, which does not need to specify rough integrals, or in the sense of Gubinelli \cite{gubinelli} with rough integrals. The following assumptions are imposed for coefficient functions.
		
(${\textbf H}_f$) $f:\R^d \to \R^d$ is locally Lipschitz continuous.\\

(${\textbf H}_g$) $g$ is in $C^3_b(\R^d,\R^{d\times m})$ where we define
\begin{equation}\label{gcond.new}
	C_g := \max \Big\{\|g\|_\infty,\|Dg\|_\infty,\|D^2g\|_\infty, \|D^3g\|_\infty\Big\}.
\end{equation}

	(${\textbf H}_X$) For a given $\nu \in (\frac{1}{3},\frac{1}{2}]$, $\R^m \ni X_t(\omega)$ is a stochastic process with stationary increments, of which almost all realizations $x$ belong to the space $C^{\nu}(\R, \R^m)$ of $\nu-$H\"older continuous paths, such that $x$ is truly rough and can be lifted into a rough path lift $\bx = (x,\X)$ of a stochastic process $(X_\cdot(\omega),\X_{\cdot,\cdot}(\omega))$ with stationary increments, and the estimate
	\begin{equation}\label{conditionx}
		\E \Big(\|X_{s,t} \|^p +\|\X_{s,t}\|^{q}\Big)\leq C_{T,\nu} |t-s|^{p \nu },\quad \forall s,t \in [0,T]
	\end{equation}
	holds for any $[0,T]$, with $p\nu \geq 1, q = \frac{p}{2}$ and some constant $C_{T,\nu}$. (Examples of such processes include multi-dimensional fractional Brownian motions).
		
	System \eqref{fSDE0} is often solved with the Doss-Sussmann technique \cite{Sus78}, namely, under a transformation $y_t = \phi_{t,\tau_k}(\bx,z_t)$ where $\phi$ is the pure rough flow, there is an one-to-one correspondence between a solution $y_t$ of \eqref{fSDE0} on a certain interval $[\tau_k,\tau_{k+1}]$ 
	and a  solution $z_t$ of the associated ordinary differential equation
	\begin{equation}\label{ascoODE}
		\dot{z}_t = \Big[\frac{\partial \phi }{\partial z}(t,\tau_k,\bx,z_t)\Big]^{-1} f( \phi_{t,\tau_k}(\bx,z_t)) =(Id + \psi_t) f(z_t+\eta_t),\quad t \in [\tau_k,\tau_{k+1}],\ z_{\tau_k} = y_{\tau_k},
	\end{equation}
	such that we can control $\eta\in \R^d$ and $\psi \in \R^{d \times d}$ so that
	\begin{equation}\label{HK}
		\|\eta_t\|,	\|\psi_t\| \leq \lambda,\quad \forall t \in [\tau_k,\tau_{k+1}].
	\end{equation}
However, estimating the solution norms for \eqref{ascoODE} is a challenge which requires a technical condition that is difficult to be applied, namely the drift has to be of linear growth in the perpendicular direction (see e.g. \cite{duckloeden}). Another difficulty is to approximate the solution of \eqref{fSDE0} using the Euler scheme 
	\begin{equation}\label{Euler}
				y_{t_{k+1}} = y_{t_k} + f(y_{t_k}) (t_{k+1}-t_k)  + g(y_{t_k})x_{t_k, t_{k+1}} + Dg(y_{t_k})g(y_{t_k})\X_{t_k,t_{k+1}},\quad k \in \N.
	    \end{equation}
In \cite{duckloeden}, it is proved that the scheme \eqref{Euler} approximates the continuous system \eqref{fSDE0} in the pathwise sense, using the cut-off technique. A drawback of this method is that, while the approximation error is estimated as $C(\bx)|\Pi|^{3\nu-1}$, the constant $C(\bx)$ might not be integrable in general; this makes it difficult to prove the approximation of the numerical solution of \eqref{Euler} to the continuous solution of \eqref{fSDE0} in $L^1$. Later, it is proved in \cite{congduchong23} that $C(\bx) \in L^1$ under an additional assumption on the global Lipschitz continuity of the drift $f$.

Recently, another condition has been proposed in \cite{ducjost25}, which introduces the new notion of a {\it strong Lyapunov function} $V$ with the existence of a small constant $\lambda_0 \in (0,1)$ and constants $C_{\lambda_0}>0, \delta \in \R$ such that
		\begin{equation}\label{gradientnegativity1}
			\sup \limits_{\substack{\psi \in \R^{d\times d},\|\psi\|\leq \lambda_0 \\\eta \in \R^d,\|\eta\| \leq \lambda_0}} \langle \nabla V(z), (I + \psi) f(z+\eta)\rangle \leq C_{\lambda_0} +\delta V(z),\quad \forall z \in \R^d.
		\end{equation}  
Under condition \eqref{gradientnegativity1}, the solution norm of \eqref{fSDE0} is estimated via the Lyapunov function, and its integrability is also proved, see \cite[Theorem 9]{ducjost25}. Moreover, under the negative gradient condition with $\delta <0$, one can prove the existence of the global random pullback attractor for the random dynamical system $\Phi$ generated from \eqref{fSDE0}. As such, the condition is well applied to many non-linear drifts such as Fitzhugh-Nagumo neuro and Lorenz systems.  

Following the setting in \cite{ducjost25}, this paper aims to close the gap on the numerical scheme above by instead considering the tamed Euler scheme
\begin{equation}\label{tamedEulerintro}
				y_{t_{k+1}} = y_{t_k} + \frac{f(y_{t_k}) (t_{k+1}-t_k)}{1+ M\|f(y_{t_k})\| (t_{k+1}-t_k)} + g(y_{t_k})x_{t_k, t_{k+1}} + Dg(y_{t_k})g(y_{t_k})\X_{t_k,t_{k+1}},\quad k \in \N
        \end{equation} 
for a parameter $M>0$. An advantage of the tamed scheme \eqref{tamedEulerintro} is that it enables us to prove the integrability of its solution norm estimate under the condition \eqref{gradientnegativity1}, thus proving the convergence of the scheme in both the pathwise and the $L^1$ senses. Moreover, under the negative gradient condition, we can prove the existence of the global numerical pullback attractor for the discrete random dynamical system generated by \eqref{tamedEulerintro} for a regular grid $\Pi^\Delta$ with step size $\Delta$. The upper semi-continuity of the numerical attractor w.r.t. the step size to the continuous attractor and the upper semi-continuity w.r.t. the noise intensity $C_g$ to the numerical attractor of the unperturbed drift are also proved in both the pathwise and the $L^1$ senses. 
        
\section{Strong Lyapunov functions}

           \subsection{$\mathcal{K}$ spaces}\label{apendix1}
       Following \cite{ducjost25}, recall that a function $\kappa: \R_+ \to \R_+$ is a $\mathcal{K}^\prime_\infty$ function if it is continuous, strictly increasing and $\lim \limits_{t \to \infty} \kappa(t) = \infty$. A function $\kappa \in \mathcal{K}^\prime_\infty$ is called a $\mathcal{K}_\infty$ function if, in addition, $\kappa(0)=0$. 
    A function $\kappa \in \mathcal{K}^\prime_\infty$ is called a $\mathcal{K}^{\rm tempered}_\infty$ function if it satisfies 
    \begin{equation}\label{tempered}
        \limsup \limits_{\epsilon \to 0} \limsup \limits_{t \to \infty}  \frac{1}{t} \log \kappa(e^{\epsilon t}) =0.
    \end{equation}
    We call a function  $\kappa \in \mathcal{K}^\prime_\infty$ a $\mathcal{K}^{\rm poly}_\infty$ if there exists constants $C_\kappa,\rho_\kappa >0$ such that 
    \begin{equation}\label{polyest}
        \kappa (t) \leq C_\kappa (1+t^{\rho_\kappa}),\quad \forall t \in \R_+.
    \end{equation}
    Note that if $\kappa \in \mathcal{K}^\prime_\infty$ then so is its inverse function $\kappa^{-1}$. In addition, it is proved in \cite{ducjost25} that $
    \mathcal{K}^{\rm poly}_\infty \subset \mathcal{K}^{\rm tempered}_\infty$ and 
    \begin{lemma}\label{Kfunctions}
If $\alpha, \beta \in \mathcal{K}^{\rm tempered}_\infty$ then so does $\beta \circ \alpha$. Also, if $\alpha, \beta \in \mathcal{K}^{\rm poly}_\infty$ then so does $\beta \circ \alpha$.
    \end{lemma}
 
    The following assumption is imposed for the strong Lyapunov functions.\\
	
	(${\textbf H}_V$) There exists for a strong Lyapunov-typed function $V \in C^1(\R^d,\R_+)$ for the drift $f$ that satisfies 
	\begin{itemize}
		\item there exists functions $\alpha,\beta \in \mathcal{K}^{\rm tempered}_\infty$ such that $\alpha^{-1}\in \mathcal{K}^{\rm tempered}_\infty$ and
		\begin{equation}\label{Vbound} 
			\alpha(\|z\|) \leq V(z) \leq \beta(\|z\|),\quad \forall z \in \R^d;	
		\end{equation}
		\item there exists a constant $L_V > 0$ such that
		\begin{equation}\label{lipschitzV}
			\| \nabla V(z)\| \leq L_V,\quad \forall z \in \R^d;
		\end{equation}
		\item there exists a constant $\lambda_0 \in (0,1)$ associated with parameters $C_{\lambda_0}, \delta >0$ such that
		\begin{equation}\label{gradientnegativity}
			\sup \limits_{\substack{\psi \in \R^{d\times d},\|\psi\|\leq \lambda_0 \\\eta \in \R^d,\|\eta\| \leq \lambda_0}}  \langle \nabla V(z), (I + \psi) f(z+\eta)\rangle \leq C_{\lambda_0} + \delta V(z),\quad \forall z \in \R^d.
		\end{equation}
	\end{itemize}
	\begin{remark}\label{remLgamma}
    i, In general, condition \eqref{gradientnegativity} can be written in the form
    \begin{equation}\label{gradientgeneral}
        \sup \limits_{\substack{\psi \in \R^{d\times d},\|\psi\|\leq \lambda_0 \\\eta \in \R^d,\|\eta\| \leq \lambda_0}}  \langle \nabla V(z), (I + \psi) f(z+\eta)\rangle \leq \gamma(V(z)),\quad \forall z \in \R^d,
    \end{equation}
    for a certain function $\gamma$ that is one-sided globally Lipschitz continuous w.r.t. the Lipschitz constant $L_\gamma$. As seen in \cite[Theorem 9]{ducjost25}, \eqref{gradientgeneral} is enough to prove the existence and uniqueness theorem. However, in this paper, we consider only the two simplest cases $\gamma (u) = C_{\lambda_0} +\delta u$ and $\gamma (u) = C_{\lambda_0} -\delta u$.\\
          
ii, In case of additive noise, i.e. $g(\cdot) \equiv \bar{g}\in \cL(\R^m,\R^d)$ is a constant matrix, it follows that 
\begin{equation}\label{etaadditive}
\eta_t = \bar{g}x_{\tau_k,t},\quad \forall t\in [\tau_k,\tau_{k+1}];
\end{equation}
thus $\psi \equiv 0$ and condition \eqref{gradientnegativity} can be modified into a simpler form
		\begin{equation}\label{addgradientnegativity}
			\sup \limits_{\eta \in \R^d,\|\eta\| \leq \lambda_0}  \langle \nabla V(z), f(z+\eta)\rangle \leq C_{\lambda_0} + \delta V(z),\quad \forall z \in \R^d,
		\end{equation}
        for a certain constant $\lambda_0 \in (0,1)$ associated with parameters $C_{\lambda_0}, \delta >0$.\\

    iii, For the deterministic (unperturbed) system $\dot{y}=f(y)$, the condition \eqref{gradientnegativity} is reduced to
    \begin{equation}\label{detergradientnegativity}
			\langle \nabla V(z), f(z)\rangle \leq C + \delta V(z),\quad \forall z \in \R^d.
		\end{equation}
        In Section \ref{tamesec}, we show in Remark \ref{detertame} that condition \eqref{detergradientnegativity} is sufficient to prove the convergence of a tamed numerical scheme to approximate the solution of the continuous system $\dot{y}=f(y)$.
	\end{remark}	
       
  \section{A tamed numerical scheme}\label{tamesec}
   In this section, we study a tamed method for the discrete system, i.e. we consider the following system 
		\begin{equation}\label{tamedEuler}
			\begin{split}
				y_0 &\in \R^d,\\
				y_{t_{k+1}} &= y_{t_k} + \frac{f(y_{t_k}) (t_{k+1}-t_k)}{1+ M\|f(y_{t_k})\| (t_{k+1}-t_k)} + g(y_{t_k})x_{t_k, t_{k+1}} + Dg(y_{t_k})g(y_{t_k})\X_{t_k,t_{k+1}},\quad k \in \N
			\end{split}
        \end{equation}
where $M >0$ is a parameter that will be specified later and $\Pi = \{t_k\}_{k \in \N}$ is the discrete time set with its resolutions 
\begin{equation}
\begin{split}
|\Pi| &:= \sup \{|t_{k+1}-t_k|: k \in \N\} \in (0,1),\\
|\Pi(I)| &:= \sup \{|t_{k+1}-t_k|: t_k \in \Pi(I)\} \in (0,1),\quad \forall \Pi(I) \subset \Pi.    
\end{split}    
\end{equation}

The Euler scheme has been studied recently in \cite{duckloeden} for one sided Lipschitz drift $f$ with the convergence rate of polynomial $|\Pi|^{3\alpha-1}$ and also dependent of the solution and the noise itself. The results in \cite{congduchong23} confirm that the convergence rate is independent of the solution norm and the noise under the assumption that the drift $f$ is globally Lipschitz continuous. 

In this section, we extend the Doss-Sussmann technique to control the solution growth for the discrete time set $\Pi$. Consider the pure rough difference equation
\begin{equation}\label{pureRDE}
    \phi_{t_{k+1}} =  \phi_{t_k} + g(\phi_{t_k})x_{t_k, t_{k+1}} + Dg(\phi_{t_k})g(\phi_{t_k})\X_{t_k,t_{k+1}},\quad k \in \N.
\end{equation}

\begin{lemma}\label{lempureRDE}
Let $\phi,\bar{\phi}$ be the solutions of the rough difference equation \eqref{pureRDE}.
    For any $s,t \in \Pi, s< t$ such that 
    \begin{equation}\label{lambdathreshold}
    C_pC_g \ltn \bx \rtn_{\tp,\Pi[s,t]} \leq \lambda \leq \frac{1}{8},
    \end{equation}
    the following estimates hold 
        \begin{eqnarray}\label{pureRDEest1}
            &i,& \|\phi_t-\phi_s\| \leq \ltn \phi\rtn_{\tp,\Pi[s,t]} \leq 4\lambda; \label{pureRDEest1}\\
            &ii,& \| (\phi_t-\bar{\phi}_t)-(\phi_s-\bar{\phi}_s)\|\leq \ltn \phi-\bar{\phi}\rtn_{\tp,\Pi[s,t]} \leq 4\lambda\Big(\|\phi_s -\bar{\phi}_s\| \wedge \|\phi_t -\bar{\phi}_t\|\Big). \label{pureRDEest2}
        \end{eqnarray}
\end{lemma}
\begin{proof}
   The proof goes line by line with the proof of \cite[Proposition 2.1]{duchongcong26}, with a small modification that all estimates of the $p$-variation norms and the application of the sewing lemma are for the discrete time set $\Pi$.
   
\end{proof}
Define the two parameter flow $\varphi(t,s,\bx)$ for $s,t\in \Pi, s\leq t$ by: $\varphi(s,s,\bx) = Id$ and
\begin{equation}
    \begin{split}
        \varphi(t,s,\bx) &=\varphi(t,t_{m-1},\bx)\circ \dots \circ \varphi(t_{l+1},s,\bx),\quad \forall s= t_l < \ldots < t_m =t;\\ 
        \varphi(t_{k+1},t_k,\bx)\phi&= \phi_{t_k} + g(\phi)x_{t_k, t_{k+1}} + Dg(\phi)g(\phi)\X_{t_k,t_{k+1}},\quad \forall k \in \N.
    \end{split}
\end{equation}
Then $\varphi(t,s,\bx)$ is a $C^1$ map. We show that
\begin{lemma}\label{lempureRDE2}
Assume that $s,t\in \Pi, s<t$ satisfies \eqref{lambdathreshold}. Then $\varphi(t_k,s,\bx)$ is invertible and $\varphi(t_k,s,\bx)^{-1}$ is $C^1$ for all $t_k \in \Pi[s,t]$. Moreover, we can write for any $\phi,\bar{\phi} \in \R^d$ a form
\begin{equation}\label{est2}
    \varphi(t_k,s)^{-1} \phi -\varphi(t_k,s)^{-1} \bar{\phi} = \Big[I +\psi_{t_k}\Big](\phi-\bar{\phi})
\end{equation}
where $\psi_{t_k} \in \R^{d\times d}$ satisfies 
\begin{equation}\label{psiest}
\psi_s = 0;\quad \|\psi_{t_k}\| \leq 4 \lambda \quad \forall s \leq t_k \leq t.
\end{equation}
\end{lemma}
\begin{proof}
For any two consecutive times $t_k,t_{k+1} \in \Pi[s,t]$, it follows from $C_pC_g \| \bx_{t_k,t_{k+1}}\| =C_pC_g \ltn \bx \rtn_{\tp,[t_k,t_{k+1}]} \leq \lambda \leq \frac{1}{8} $ that the derivative
\[
D\varphi (t_{k+1},t_k,\bx)\phi = Id + Dg(\phi)x_{t_k, t_{k+1}} + \Big(D^2g(\phi)g(\phi) + Dg(\phi)Dg(\phi)\Big)\X_{t_k,t_{k+1}}
\]
satisfies
\[
\min_{\|\xi\|=1} \left\|\Big[D\varphi (t_{k+1},t_k,\bx)\phi\Big]\xi\right\| \geq 1 - C_g\|x_{t_k, t_{k+1}}\| - 2C_g^2\|\X_{t_k,t_{k+1}}\| \geq 1-3\lambda >0;
\]
therefore, it is invertible. Due to the implicit function theorem, $\varphi(t_{k+1},t_k,\bx)$ is an invertible map and $\varphi(t_{k+1},t_k,\bx)^{-1}$ is also a $C^1$ map. As a consequence, $\varphi(t_k,s)$ is invertible for any $t_k \in \Pi[s,t]$ and $\varphi(t_k,s)^{-1}$ is also a $C^1$ map. Moreover, from \eqref{pureRDEest2} it follows that
\begin{equation}\label{est1}
\|\varphi(t_k,s)^{-1} \phi -\varphi(t_k,s)^{-1} \bar{\phi} - (\phi-\bar{\phi})\| \leq 4 \lambda \|\phi-\bar{\phi}\|. 
\end{equation}
Applying the Lagrange mean value theorem to \eqref{est1}, we obtain
\begin{equation*}
    \varphi(t_k,s)^{-1} \phi -\varphi(t_k,s)^{-1} \bar{\phi} = \int_0^1 D\varphi(t_k,s)^{-1}\Big(\bar{\phi}+\theta(\phi-\bar{\phi})\Big)(\phi-\bar{\phi}) d\theta = \Big[I +\psi_{t_k}\Big](\phi-\bar{\phi})
\end{equation*}
which shows \eqref{est2}, where $\psi_{t_k}$ satisfies \eqref{psiest} due to \eqref{est1}.
\end{proof}
From now on, we choose $\lambda$ and $M$ so that
\begin{equation}\label{Mest}
  \lambda \leq \frac{1}{8} \quad \text{and}\quad  4\lambda + \frac{1+4\lambda}{M} \leq \lambda_0.
\end{equation}
\begin{proposition}\label{case1}
    Under the assumptions (${\textbf H}_{V}$), (${\textbf H}_{g}$) and \eqref{Mest}, assume that $s, t \in \Pi, s < t$ are two discrete times satisfying \eqref{lambdathreshold}. Then the following estimate holds
\begin{equation}\label{Vest4}
    V(y_t) +\bar{C}\leq \Big[V(y_s)+\bar{C}\Big] e^{\delta (t-s)} + 4L_V \lambda.
\end{equation}
where
\begin{equation}\label{Cbar}
    \bar{C} :=\Big(\frac{C_{\lambda_0}}{\delta} + L_V \lambda_0\Big).
\end{equation}
\end{proposition}
\begin{proof}
  We introduce the discrete Doss-Sussmann transformation 
\begin{equation}\label{DS1}
y_{t_k} = \varphi(t_k,s,\bx)z_{t_k} = z_{t_k} + \eta_{t_k},\quad \forall s \leq t_k \leq t.      
\end{equation}
Then $z_{t_k} = \varphi(t_k,s,\bx)^{-1}y_{t_k}$, in particular $z_s = y_s$. Due to \eqref{pureRDEest1},
\begin{equation}\label{etaest}
\eta_s = 0;\quad \|\eta_{t_k}\| \leq 4\lambda \quad \forall s \leq t_k \leq t.
\end{equation}
Replacing \eqref{DS1} into \eqref{tamedEuler}, we obtain
\begin{eqnarray*}
\varphi(t_{k+1},s,\bx)z_{t_{k+1}} &=& \varphi(t_{k+1},t_k,\bx) \circ \varphi(t_k,s,\bx)z_{t_k} + \frac{f(\varphi(t_k,s,\bx)z_{t_k}) (t_{k+1}-t_k)}{1+ M\|f(\varphi(t_k,s,\bx)z_{t_k})\| (t_{k+1}-t_k)}\\
&=& \varphi(t_{k+1},s,\bx)z_{t_k} + \frac{f(z_{t_k}+\eta_{t_k}) (t_{k+1}-t_k)}{1+ M\|f(z_{t_k}+\eta_{t_k})\| (t_{k+1}-t_k)}.    
\end{eqnarray*}
Taking into account the invertibility of $\varphi(t_{k+1},s,\bx)$, we obtain
\begin{equation}\label{transformed1}
    z_{t_{k+1}} = \varphi(t_{k+1},s,\bx)^{-1}\Big(\varphi(t_{k+1},s,\bx)z_{t_k} + \frac{f(z_{t_k}+\eta_{t_k}) (t_{k+1}-t_k)}{1+ M\|f(z_{t_k}+\eta_{t_k})\| (t_{k+1}-t_k)}\Big).
\end{equation}
Applying \eqref{est2} yields
\begin{eqnarray}\label{transformed2}
  \Delta z_k=  z_{t_{k+1}} -z_{t_k} &=& \varphi(t_{k+1},s,\bx)^{-1}\Big(\varphi(t_{k+1},s,\bx)z_{t_k} + \frac{f(z_{t_k}+\eta_{t_k}) (t_{k+1}-t_k)}{1+ M\|f(z_{t_k}+\eta_{t_k})\| (t_{k+1}-t_k)}\Big) \notag\\
    && - \varphi(t_{k+1},s,\bx)^{-1} \Big(\varphi(t_{k+1},s,\bx)z_{t_k} \Big) \notag\\
    &=& \Big[I +\psi_{t_k}\Big]\Big(\frac{f(z_{t_k}+\eta_{t_k}) (t_{k+1}-t_k)}{1+ M\|f(z_{t_k}+\eta_{t_k})\| (t_{k+1}-t_k)}\Big), \quad \forall s \leq t_k < t.
\end{eqnarray}
Because of \eqref{transformed2} and \eqref{psiest},
\begin{equation}\label{zest}
    \|\Delta z_k\| \leq (1+4\lambda) \frac{\|f(z_{t_k}+\eta_{t_k})\| (t_{k+1}-t_k)}{1+ M\|f(z_{t_k}+\eta_{t_k})\| (t_{k+1}-t_k)} \leq \frac{1+4\lambda}{M}, \quad \forall s \leq t_k <t.
\end{equation}
We would like now to estimate the Lyapunov-typed function $V$ on the transformed variables $z$. From \eqref{transformed2} and the Lagrangian mean value theorem, there exists $\theta^* \in (0,1)$ such that
\begin{eqnarray}\label{Vest1}
    V(z_{t_{k+1}}) &=& V(z_{t_k}) + \Big\langle \nabla V(z_{t_k}+ \theta^* \Delta z_k), \Big[I +\psi_{t_k}\Big]\Big(\frac{f(z_{t_k}+\eta_{t_k}) (t_{k+1}-t_k)}{1+ M\|f(z_{t_k}+\eta_{t_k})\| (t_{k+1}-t_k)}\Big)\Big\rangle \notag\\
    &=& V(z_{t_k}) + \frac{(t_{k+1}-t_k)}{1+ M\|f(z_{t_k}+\eta_{t_k})\| (t_{k+1}-t_k)} \times \notag\\
    &&\times \Big\langle \nabla V(z_{t_k}+ \theta^* \Delta z_k), \Big[I +\psi_{t_k}\Big]f\Big(z_{t_k}+\theta^* \Delta z_k+\eta_{t_k}-\theta^* \Delta z_k\Big) \Big\rangle 
\end{eqnarray}
From \eqref{etaest}, \eqref{zest} and taking into account \eqref{Mest},
\[
\|\eta_{t_k} - \theta^* \Delta z_k\| \leq \|\eta_{t_k}\|+|\theta^*|\|\Delta z_k\|\leq 4 \lambda + \frac{1+4\lambda}{M} \leq \lambda_0, \quad \forall s \leq t_k < t,\theta^* \in (0,1). 
\]
As a result, we apply \eqref{gradientnegativity} and \eqref{lipschitzV} for $z:= z_{t_k}+\theta^* \Delta z_k$ to obtain for all $s \leq t_k <t$ the estimate 
 \allowdisplaybreaks
\begin{eqnarray}\label{Vest2}
    V(z_{t_{k+1}}) &\leq& V(z_{t_k}) + \frac{(t_{k+1}-t_k)}{1+ M\|f(z_{t_k}+\eta_{t_k})\| (t_{k+1}-t_k)} \Big(C_{\lambda_0} +\delta V(z_{t_k}+ \theta^* \Delta z_k)\Big) \notag\\
    &\leq& V(z_{t_k}) + \Big[C_{\lambda_0} + L_V \delta \lambda_0 + \delta V(z_{t_k})\Big] (t_{k+1}-t_k)\notag\\
    &\leq& \Big[\bar{C}+V(z_{t_k})\Big] \Big[1+\delta (t_{k+1}-t_k)\Big] - \bar{C}\notag\\
    &\leq& \Big[\bar{C}+V(z_{t_k})\Big] e^{\delta (t_{k+1}-t_k)} - \bar{C}.
\end{eqnarray}
By induction, it is easy to show that from \eqref{Vest2} that
\begin{equation}\label{Vest3}
    V(z_{t_k}) +\bar{C}\leq \Big[V(z_s)+\bar{C}\Big] e^{\delta (t_k-s)},\quad \forall s \leq t_k \leq t.
\end{equation}
Hence, by \eqref{lipschitzV} and \eqref{etaest}, we obtain
\begin{equation*}
      V(y_{t_k})  \leq V(z_{t_k}) + 4L_V\lambda\leq \Big[V(y_s)+\bar{C}\Big] e^{\delta (t_k-s)}-\bar{C} + 4L_V \lambda ,\quad \forall s \leq t_k \leq t.
  \end{equation*}
which proves \eqref{Vest4}.\\
\end{proof}
\begin{proposition}\label{case2}
    Under assumptions (${\textbf H}_{V}$), (${\textbf H}_{g}$), assume $s,t \in \Pi, s <t$ are two consecutive times in $\Pi$. Then the following estimate holds
\begin{equation}\label{Vest5}
    V(y_t) +\bar{C} \leq [V(y_s)+\bar{C}]e^{\delta(t-s)}+  L_V\Big(C_g \| x_{s,t}\| + C_g^2 \| \X_{s,t}\|^2\Big).
\end{equation}
\end{proposition}
\begin{proof}
    The proof is straightforward from \eqref{lipschitzV} and the tamed equation \eqref{tamedEuler}. Indeed, 
    \begin{eqnarray*}
        V(y_t) &\leq& V\Big(y_s+\frac{f(y_s) (t-s)}{1+ M\|f(y_s)\| (t-s)}\Big) + L_V\Big(  C_g \| x_{s,t}\| + C_g^2 \| \X_{s,t}\|\Big)\\
        &\leq& V(y_s) +\Big\langle \nabla V\Big(y_s+\theta^*\frac{f(y_s) (t-s)}{1+ M\|f(y_s)\| (t-s)}\Big),f(y_s)\Big\rangle\frac{ (t-s)}{1+ M\|f(y_s)\| (t-s)} \\
        &&+ L_V\Big( C_g \| x_{s,t}\| + C_g^2 \| \X_{s,t}\|\Big)\\
        &\leq& V(y_s) + \Big[C_{\lambda_0} + \delta V\Big(y_s+\theta^*\frac{f(y_s) (t-s)}{1+ M\|f(y_s)\| (t-s)}\Big) \Big](t-s)\\
        &&+ L_V\Big(C_g \| x_{s,t}\| + C_g^2 \| \X_{s,t}\|\Big)\\
        &\leq& V(y_s) + \Big(C_{\lambda_0} + \delta V(y_s) + \delta L_V \frac{1}{M} \Big)(t-s)+ L_V\Big(C_g \| x_{s,t}\| + C_g^2 \| \X_{s,t}\|\Big)\\
         &\leq& [V(y_s) +\bar{C}]\Big[1+\delta (t-s)\Big] -\bar{C}+ L_V\Big(C_g \| x_{s,t}\| + C_g^2 \| \X_{s,t}\|\Big);
    \end{eqnarray*}
    which proves \eqref{Vest5}.\\
\end{proof}

By constructing the discrete stopping times $\Big\{\tau^\Pi_n(\frac{\lambda}{C_pC_g},\bx,\Pi(I))\Big\}$, we can now formulate the first main result of this paper as below. First, we introduce a notation of the maximum increment of $\bx$ on a set $\Pi[s,t]$
\begin{equation}\label{bxdelta}
    \|\Delta \bx\|_{\infty,\Pi[s,t]} := \max \{\|\bx_{t_k,t_{k+1}}\|: s \leq t_k < t\}.
\end{equation}

\begin{theorem}\label{tamedscheme}
    Under assumptions (${\textbf H}_{V}$), (${\textbf H}_{g}$) and \eqref{Mest}, the solution $y$ of the tamed numerical scheme \eqref{tamedEuler} satisfies
    \begin{equation}\label{Vestmain}
        V(y_t) +\bar{C}
        \leq  e^{\delta t} \big[V(y_0) +\bar{C} + H\big(\bx,\Pi[0,t]\big)\big],\quad \forall t\in \Pi
    \end{equation}    
 where
 \begin{equation}\label{Hbx}
 H\big(\bx,\Pi[0,t]\big)= L_V\Big(4\lambda +C_g \|\Delta \bx\|_{\infty,\Pi[0,t]} + C_g^2 \|\Delta \bx\|_{\infty,\Pi[0,t]}^2\Big)N\Big(\frac{\lambda}{C_pC_g},\bx,\Pi[0,t]\Big).
    \end{equation}
\end{theorem}
\begin{proof}
  A direct consequence of the estimates \eqref{Vest4} in Propositions \ref{case1} and \eqref{Vest5} in Proposition \ref{case2} shows that: for a sequence of stopping times $\Big\{\tau^\Pi_n(\frac{\lambda}{C_pC_g},\bx,\Pi([0,t]))\Big\}$, we obtain the overall estimate
  \[
  V(y_{\tau^\Pi_{n+1}}) +\bar{C} \leq e^{\delta(\tau^\Pi_{n+1}-\tau^\Pi_n)} \Big[V(y_{\tau^\Pi_n}) +\bar{C}\Big] + L_V\big(4\lambda + C_g \|\Delta \bx\|_{\infty,\Pi[0,t]} + C_g^2 \|\Delta \bx\|_{\infty,\Pi[0,t]}^2\big)
  \]
  for any two consecutive stopping times in $\Big\{\tau^\Pi_n(\frac{\lambda}{C_pC_g},\bx,\Pi([0,t]))\Big\}$.
The rest of the arguments are followed by induction, similar to the proof of \cite[Theorem 9]{ducjost25}.\\
\end{proof}
\begin{remark}
i, Similar to the proof of \cite[Theorem 9]{ducjost25}, we can obtain a general estimate of \eqref{Vestmain} corresponding to the general gradient condition \eqref{gradientgeneral} for the strong Lyapunov function $V$.\\

ii, It follows from \eqref{Nestdiscrete} that $H$ in \eqref{Hbx} can be bounded from above by $H \leq \bar{H}$, where 
    \begin{equation}\label{Hupper}
 \bar{H}\big(\bx,[0,t]\big)= 2L_V\Big(4\lambda + C_g \ltn\bx\rtn_{\tp,[0,t]} + C_g^2 \ltn\bx\rtn_{\tp,[0,t]}^2\Big)N\Big(\frac{\lambda}{C_pC_g},\bx,[0,t]\Big)
    \end{equation}
   is independent of $\Pi$.\\
   
iii, It is proved in \cite[Theorem 9]{ducjost25} that the solution $\Phi(t,\bx,y_0)$ of the rough differential equation \eqref{fSDE0} satisfies a similar estimate to \eqref{Vestmain}, i.e.
    \begin{equation}\label{Vestcont}
         V(y_t) +\frac{C_{\lambda_0}}{\delta} 
        \leq  e^{\delta t} \Big[V(y_0) +\frac{C_{\lambda_0}}{\delta} + H^*\big(\bx,[0,t]\big)\Big],\quad \forall t \geq 0
    \end{equation}    
 where
 \begin{equation}\label{H*bx}
 H^*\big(\bx,[0,t]\big)= L_V\lambda N\Big(\frac{\lambda}{C_pC_g},\bx,[0,t]\Big).
    \end{equation}
 Then, it follows from \eqref{Vestmain}, \eqref{Vestcont} and \eqref{Vbound} that
    \begin{equation}\label{solest}
    \begin{split}
        \|y_\cdot(\bx,y_0)\|_{\infty,\Pi[0,T]} \leq \quad & \alpha^{-1}\Big(e^{\delta T} \Big[\beta(\|y_0\|) +\frac{C_{\lambda_0}}{\delta} + L_V \lambda_0 + \bar{H}\big(\bx,[0,T]\big)\Big] \Big);\\
        \|\Phi(\cdot,\bx)y_0\|_{\infty,[0,T]} \leq \quad & \alpha^{-1}\Big(e^{\delta T} \Big[\beta(\|y_0\|) +\frac{C_{\lambda_0}}{\delta} + H^*\big(\bx,[0,T]\big)\Big] \Big).
        \end{split}
    \end{equation}
    
\end{remark}
\begin{theorem}\label{tameapprox}
    i, Under assumptions (${\textbf H}_{V}$), (${\textbf H}_{g}$) and \eqref{Mest}, the tamed numerical scheme \eqref{tamedEuler} approximates the solution $\Phi(\cdot,\bx)y_0$ of the rough differential equation \eqref{fSDE0} in the pathwise sense, i.e. for any $T>0$, there exists a random variable $C(T,\bx,y_0)$ (not necessarily integrable) such that for sufficiently small $|\Pi|$ (dependent on $\bx(\omega)$),
    \begin{equation}\label{LPi}
    \mathbf{L}(T,\Pi,\bx,\|y_0\|) :=     \|y_\cdot(\bx,y_0) - \Phi(\cdot,\bx)y_0\|_{\infty,\Pi[0,T]} \leq C(T,\bx,\|y_0\|) |\Pi|^{3\nu -1}.
    \end{equation}
    
    ii, Under assumptions (${\textbf H}_{V}$) with $\alpha^{-1},\beta \in \mathcal{K}^{\rm poly}_\infty$, (${\textbf H}_{g}$) and \eqref{Mest},
    the tamed numerical scheme \eqref{tamedEuler} with regular grid $\Pi^\Delta$ approximates the solution $\Phi(\cdot,\bx)y_0$ of the rough differential equation \eqref{fSDE0} in the sense that: for any $\rho \geq 1$ and $T>0$, 
     \begin{equation}\label{LdeltaExpec}
     \lim \limits_{\Delta \to 0}   \Big\| \mathbf{L}(T,\Delta,\bx(\cdot),\|y_0\|)\Big\|_{\mathcal{L}^\rho} =0.
    \end{equation}
\end{theorem}
\begin{proof}
i, The proof follows \cite[Theorem 3.1]{duckloeden} line by line, except for a small modification. Namely, we first use \eqref{solest} and the cut-off technique to work with the cut-off drift $f_R$ which is bounded by $\|f_R\|_\infty$ in the closed ball $B(0,R)$ where $R = R(T,\bx,y_0)$. Then we study the tamed numerical method \eqref{tamedEuler} for coefficient functions $f_R$ and $g$ and estimate the difference between the solution $\Phi^R(\cdot,\bx)y_0$ of the continuous rough differential equation \eqref{fSDE0} with $f_R$ and the solution $y^R_\cdot(\bx,y_0)$ of the tamed numerical scheme \eqref{tamedEuler} with $f_R$, using the triangle inequality
\begin{eqnarray*}
\|\Phi^R(t_m, 0)y_0 - y^R_{t_m}(\bx,y_0)\| &\leq& \sum_{k=0}^{m-1} \|\Phi^R(t_m, t_k)y^R_{t_k} - \Phi^R(t_m,t_{k+1},y^R_{t_{k+1}})\| \\
&\leq& C(R,\bx,T)\sum_{k=0}^{m-1} \|\Phi^R(t_{k+1}, t_k)y^R_{t_k} - y^R_{t_{k+1}}\|,    
\end{eqnarray*}
for a generic constant $C(T,\bx, R)$. The crucial estimate that is different from \cite[Theorem 3.1]{duckloeden} is 
 \allowdisplaybreaks
\begin{eqnarray*}
&& \|\Phi^R(t_{k+1}, t_k)y^R_{t_k} - y^R_{t_{k+1}}\| \\
&\leq& \int_{t_k}^{t_{k+1}} \Big \|f_R(\Phi^R(u,t_k)y^R_{t_k}) - \frac{f_R(y^R_{t_k})}{1+M\|f_R(y^R_{t_k})\|(t_{k+1}-t_k)} \Big \|du \\
&& + \Big \|\int_{t_k}^{t_{k+1}} g(\Phi^R(u,t_k)y^R_{t_k}) dx_u - g(y^R_{t_k})x_{t_k,t_{k+1}} - Dg(y^R_{t_k})g(y^R_{t_k})\X_{t_k,t_{k+1}}\Big \|\\
&\leq& \int_{t_k}^{t_{k+1}} \Big \|f_R(\Phi^R(u,t_k)y^R_{t_k}) - f_R(y^R_{t_k}) \Big \|du  + \frac{M\|f_R(y^R_{t_k})\|(t_{k+1}-t_k)^2}{1+M\|f_R(y^R_{t_k})\|(t_{k+1}-t_k)}\\
&& + C(T,\bx,R) (t_{k+1}-t_k)^{3\nu}\\
&\leq& L(f_R)(t_{k+1}-t_k) \ltn \Phi(\cdot,t_k)y^R_{t_k}\rtn_{\tp,[t_k,t_{k+1}]}  + M\|f_R\|_\infty (t_{k+1}-t_k)^2\\
&&+ C(T,\bx,R) (t_{k+1}-t_k)^{3\nu}\\
&\leq& C(T,\bx,R) \Big[(t_{k+1}-t_k)^{1+\nu} +(t_{k+1}-t_k)^2+ (t_{k+1}-t_k)^{3\nu}\Big] \leq C(T,\bx,R)(t_{k+1}-t_k)^{3\nu}
\end{eqnarray*}
for a generic constant $C(T,\bx,R)$. Using the same arguments as in \cite[Theorem 3.1]{duckloeden}, we obtain \eqref{LPi}.\\

ii,  For the regular grid $\Pi^\Delta$, it follows from \eqref{LPi} that 
\begin{equation}\label{Ldeltaest}
    \mathbf{L}(T,\Delta,\bx,\|y_0\|) :=     \|y_\cdot(\bx,y_0) - \Phi(\cdot,\bx)y_0\|_{\infty,\Pi^\Delta[0,T]} \leq C(T,\bx,\|y_0\|) \Delta^{3\nu -1},
    \end{equation}
for $\Delta$ sufficiently small (dependent on $\bx(\omega)$) which then proves that
    \begin{equation}\label{Ldelta}
     \lim \limits_{\Delta \to 0}   \mathbf{L}(T,\Delta,\bx,\|y_0\|)^\rho =0,\quad \text{almost surely in } \bx(\omega).
    \end{equation}
    On the other hand, the estimates \eqref{solest} and the assumption $\alpha^{-1},\beta \in \mathcal{K}^{\rm poly}_\infty$ then show that the right hand sides of \eqref{solest} are in $\mathcal{L}^\rho$. As a result,
    \begin{eqnarray}\label{Ldeltarho}
        \mathbf{L}(T,\Delta,\bx,\|y_0\|) &\leq& \|y_\cdot(\bx,y_0)\|_{\infty,\Pi[0,T]} + \|\Phi(\cdot,\bx)y_0\|_{\infty,[0,T]} \notag\\
        &\leq& \alpha^{-1}\Big(e^{\delta T} \Big[\beta(\|y_0\|) +\frac{C_{\lambda_0}}{\delta} + L_V \lambda_0 + \bar{H}\big(\bx,[0,T]\big)\Big] \Big) \notag\\
        &&+\alpha^{-1}\Big(e^{\delta T} \Big[\beta(\|y_0\|) +\frac{C_{\lambda_0}}{\delta} + H^*\big(\bx,[0,T]\big)\Big] \Big)\in \mathcal{L}^\rho.
    \end{eqnarray}
    Hence, by the Lebesgue dominated convergence theorem, \eqref{LdeltaExpec} is followed from \eqref{Ldelta} and \eqref{Ldeltarho}.
    
\end{proof}

\begin{remark}\label{addnoisetame}
    i, In the case of additive noise $g \equiv \bar{g}$ for a constant matrix $\bar{g}$, the conclusions of Theorem \ref{tameapprox} still hold under the assumptions (${\textbf H}_{V}$) with additive condition \eqref{addgradientnegativity}, and (${\textbf H}_{g}$), \eqref{Mest}. The proof goes line by line with the proof of Theorem \ref{tameapprox} by assigning $\psi \equiv 0$, hence it will be omitted here.

    ii, The $L^1$ convergence in \eqref{LdeltaExpec} is proved in an indirect way using Lebesgue dominated convergence theorem, hence it does not guarantee the convergence rate of $|\Pi|^{3\nu-1}$ as in the pathwise estimate \eqref{Ldeltaest}.
\end{remark}

\begin{remark}\label{detertame}
    Condition \eqref{addgradientnegativity} is assumed to be applicable for the deterministic (unperturbed) system $\dot{y} = f(y)$ by setting $g \equiv 0$. In fact, we can use a simplified version that
    \begin{equation}\label{deterV}
         \langle \nabla V(z), f(z)\rangle \leq C_0 + \delta V(z),\quad \forall z \in \R^d.
    \end{equation}
Consider the deterministic tamed numerical scheme on the regular grid $\Pi[0,T] = \{t_k: k =0\dots N\}$
\begin{equation}\label{tamedEulerdeter}
				z_0 \in \R^d,\quad
				z_{k+1} = z_k + \frac{f(z_k) \Delta_k}{1+ M\|f(z_k)\| \Delta_k},\quad k = 0 \dots N.
        \end{equation}
    Then a direct computation shows that
     \allowdisplaybreaks
    \begin{eqnarray*}
        V(z_{k+1}) &=& V(z_k) + \Big\langle \nabla V\Big(z_k + \theta^* \frac{f(z_k) \Delta_k}{1+ M\|f(z_k)\| \Delta_k}\Big),\frac{f(z_k) \Delta_k}{1+ M\|f(z_k)\| \Delta_k}\Big\rangle\\
        &\leq& V(z_k) + \frac{L_V}{1+ M\|f(z_k)\| \Delta_k}  \Big\|f\Big(z_k+ \theta^* \frac{f(z_k) \Delta_k}{1+ M\|f(z_k)\| \Delta_k}\Big) - f(z_k)\Big\|\Delta_k\\
        &&+ \frac{1}{1+ M\|f(z_k)\| \Delta_k} \Big\langle \nabla V\Big(z_k + \theta^* \frac{f(z_k) \Delta_k}{1+ M\|f(z_k)\| \Delta_k}\Big),f\Big(z_k+ \theta^* \frac{f(z_k) \Delta_k}{1+ M\|f(z_k)\| \Delta_k}\Big) \Big\rangle\Delta_k\\
&\leq& V(z_k) + \frac{L_V}{1+ M\|f(z_k)\| \Delta_k}  \|Df\|_{\infty,B(0,\|z_k\|+\frac{1}{M})} \frac{\|f(z_k)\| \Delta_k}{1+ M\|f(z_k)\| \Delta_k}\Delta_k\\
&&+ \frac{1}{1+ M\|f(z_k)\| \Delta_k} \Big[C_0+\delta V\Big(z_k + \theta^* \frac{f(z_k) \Delta_k}{1+ M\|f(z_k)\| }\Big)\Big]\Delta_k \\
&\leq& V(z_k) + L_V \|Df\|_{\infty,B(0,\|z_k\|+\frac{1}{M})} \|f(z_k)\|\Delta_k^2+ \Big(C_0+\frac{L_V\delta}{M} + \delta V(z_k)\Big)\Delta_k\\
&\leq& V(z_k)(1+\delta \Delta_k) + \Big(L_V+ C_0+\frac{L_V\delta}{M} \Big)\Delta_k, \quad k =0,\dots,N-1
    \end{eqnarray*}
    provided that 
 \begin{equation}\label{stepsize1}
   \max_{k =0\dots N}  \|Df\|_{\infty,B(0,\|z_k\|+\frac{1}{M})} \|f(z_k)\| |\Pi[0,T]| \leq 1.
 \end{equation}
As a result, it follows from induction that
\[
V(z_k) +\frac{1}{\delta}\Big(L_V+ C_0+\frac{L_V\delta}{M} \Big)  \leq e^{\delta t_k} \Bigg( V(z_0) + \frac{1}{\delta}\Big(L_V+ C_0+\frac{L_V\delta}{M} \Big)\Bigg),\quad k =0\dots N.
\]
Hence 
\begin{equation}\label{zmax}
    \max_{k =0\dots N}\|z_k\| \leq \alpha^{-1} \Bigg(e^{\delta T} \Bigg[ V(z_0) + \frac{1}{\delta}\Big(L_V+ C_0+\frac{L_V\delta}{M} \Big) \Bigg] \Bigg) =: \bar{M}(V(z_0))
\end{equation}
provided that \eqref{stepsize1} is satisfied by choosing $|\Pi[0,T]|$ such that
\begin{equation}\label{stepsize2}
    \|Df\|_{\infty,B\big(0,\bar{M}(V(z_0))+\frac{1}{M}\big)} \|f\|_{\infty,B\big(0,\bar{M}(V(z_0))\big)} |\Pi[0,T]| \leq 1.
\end{equation}
Similar to the proof of Theorem \ref{tameapprox} (i), the error estimate of the tamed numerical scheme in \eqref{LPi} has the form
\begin{equation}\label{LPideter}
    \mathbf{L}(T,\Pi,\bx,\|z_0\|) :=     \|z_\cdot(z_0) - \Phi(\cdot)z_0\|_{\infty,\Pi[0,T]} \leq C(T,\|z_0\|) |\Pi|.
    \end{equation}
\end{remark}

\section{Numerical attractors}
Theorem \ref{tamedscheme} and similar arguments in \cite{duckloeden} show that under condition \eqref{Mest}, the tamed numerical system \eqref{tamedEuler} with a regular grid $\Pi^\Delta$ generates a discrete random dynamical system $\Phi_{C_g}^\Delta(t,\bx)$ on $\Pi^\Delta$. In this section, we impose in assumption (${\textbf H}_{V}$) a new condition \eqref{gradientgeneral} by setting 
\begin{equation}\label{deltamain}
\gamma (u) = C_{\lambda_0} - \delta u.    
\end{equation}
Under this new condition, it is proved in \cite[Theorem 15]{ducjost25} that the continuous system \eqref{fSDE0} generates a continuous random dynamical system $\Phi_{C_g}(t,\bx)$ that admits a global random pullback attractor $\cA^0_{C_g}(\omega) \in L^1$. 
We are going to prove a similar result on the existence of the numerical attractor $\cA_{C_g}^\Delta(\omega)$ for the discrete system $\Phi_{C_g}^\Delta$.

\subsection{Deterministic attractors}\label{deterattractor}
Indeed, let us first review the simplest case of the deterministic tamed numerical scheme on the regular grid $\Pi^\Delta$, i.e.
\begin{equation}\label{tamedEulerdeter}
				z_0 \in \R^d,\quad
				z_{k+1} = z_k + \frac{f(z_k) \Delta}{1+ M\|f(z_k)\| \Delta},\quad k \in \N,
        \end{equation}
where we write $z_k$ to indicate $z_{t_k}$ for simplicity. We impose the simplified version that
         \begin{equation}\label{deterVneg}
         \langle \nabla V(z), f(z)\rangle \leq C_0 - \delta V(z),\quad \forall z \in \R^d.
    \end{equation}
Then similar to the previous part, a direct computation shows that
 \allowdisplaybreaks
    \begin{eqnarray*}
        V(z_{k+1}) &=& V(z_k) + \Big\langle \nabla V(z_k + \theta^* \frac{f(z_k) \Delta}{1+ M\|f(z_k)\| \Delta}),\frac{f(z_k) \Delta}{1+ M\|f(z_k)\| \Delta}\Big\rangle\\
        &\leq& V(z_k) + \frac{L_V}{1+ M\|f(z_k)\| \Delta}  \|f\big(z_k+ \theta^* \frac{f(z_k) \Delta}{1+ M\|f(z_k)\| \Delta}\big) - f(z_k)\|\Delta\\
        &&+ \frac{1}{1+ M\|f(z_k)\| \Delta} \Big\langle \nabla V(z_k + \theta^* \frac{f(z_k) \Delta}{1+ M\|f(z_k)\| \Delta}),f\big(z_k+ \theta^* \frac{f(z_k) \Delta}{1+ M\|f(z_k)\| \Delta}\big) \Big\rangle\Delta\\
&\leq& V(z_k) + \frac{L_V}{1+ M\|f(z_k)\| \Delta}  \|Df\|_{\infty,B(0,\|z_k\|+\frac{1}{M})} \frac{\|f(z_k)\| \Delta}{1+ M\|f(z_k)\| \Delta}\Delta\\
&&+ \frac{1}{1+ M\|f(z_k)\| \Delta} \Big(C_0-\delta V\big(z_k + \theta^* \frac{f(z_k) \Delta}{1+ M\|f(z_k)\| }\big)\Big)\Delta \\
&\leq& V(z_k) + L_V\|Df\|_{\infty,B(0,\|z_k\|+\frac{1}{M})} \|f(z_k)\| \Delta^2 + \frac{\Delta}{1+ M\|f(z_k)\| \Delta} \Big(C_0-\delta V(z_k) + \delta L_V \frac{1}{M}\Big) \\
&\leq& V(z_k)\Big(1-\frac{\delta \Delta}{1+ M\|f(z_k)\| \Delta} \Big) + L_V \|Df\|_{\infty,B(0,\|z_k\|+\frac{1}{M})} \|f(z_k)\|\Delta^2+ \Big(C_0+\frac{L_V\delta}{M} \Big)\Delta\\
&\leq& V(z_k)(1-\frac{\delta \Delta}{2}) + \Big(L_V+ C_0+\frac{L_V\delta}{M} \Big)\Delta, \quad k =0\in \N
    \end{eqnarray*}
    provided that 
 \begin{equation}\label{stepsize3}
   \max_{k \in \N}  \Big\{ \|Df\|_{\infty,B(0,\|z_k\|+\frac{1}{M})} \|f(z_k)\| \Delta, \ M\|f(z_k)\| \Delta \Big\}\leq 1.
 \end{equation}
As a result, it follows from induction that
\[
V(z_k) -\frac{1}{\delta}\Big(L_V+ C_0+\frac{L_V\delta}{M} \Big)  \leq e^{-\frac{\delta}{2} k\Delta} \Big[ V(z_0) - \frac{1}{\delta}\Big(L_V+ C_0+\frac{L_V\delta}{M} \Big)\Big],\quad k \in \N.
\]
Hence 
\begin{equation}\label{zmax2}
    \sup_{k \in \N}\|z_k\| \leq \alpha^{-1} \Big( V(z_0) + \frac{1}{\delta}\Big[L_V+ C_0+\frac{L_V\delta}{M} \Big] \Big) =: M^\infty(V(z_0))
\end{equation} 
provided that \eqref{stepsize3} is satisfied by choosing $\Delta = \Delta(\|z_0\|)$ such that
\begin{equation}\label{stepsize4}
   \max \Big\{ \|Df\|_{\infty,B\big(0,M^\infty(V(z_0))+\frac{1}{M}\big)} \|f\|_{\infty,B\big(0,M^\infty(V(z_0))\big)},\ M\|f\|_{\infty,B\big(0,M^\infty(V(z_0))\big)}\Big\}  \Delta \leq 1.
\end{equation}
Therefore, we prove that $\Phi^\Delta$ admits a global numerical attractor $\cA_0^\Delta$ which lies within the compact absorbing set $\cB = B\Big(0,V^{-1}\Big(\frac{1}{\delta}\Big[L_V+ C_0+\frac{L_V\delta}{M} \Big]\Big)\Big)$. In practice, we can start with 
\[
z_0 \in B\Big(0,V^{-1}\Big(\frac{1}{\delta}\Big[L_V+ C_0+\frac{L_V\delta}{M} \Big]+1\Big)\Big);
\]
then $\Delta$ chosen in \eqref{stepsize4} is independent of $z_0$. Due to the uniform boundedness of $\cA_0^\Delta$ inside $\cB$, it is easy (see e.g. \cite{ducjost25}) to prove the semi-continuity of $\cA_0^\Delta$, i.e.
    \begin{equation}\label{numattractor}
        \lim \limits_{\Delta \to 0} d_H(\cA_0^\Delta|\cA^0_0) = 0. 
    \end{equation}

\subsection{Random attractors}
Next, we consider, under condition \eqref{Mest}, the tamed numerical scheme \eqref{tamedEuler} for the regular grid $\Pi^\Delta =\{k\Delta,k\in \N\}$, i.e.
\begin{equation}\label{tamedEulerDelta}
			\begin{split}
				y_0 &\in \R^d,\\
				y_{k+1} &= y_k + \frac{f(y_k) \Delta}{1+ M\|f(y_k)\| \Delta} + g(y_k)\Delta x_k + Dg(y_k)g(y_k)\Delta \X_k,\quad k \in \N,
			\end{split}
        \end{equation}
 where we write in short $y_k = y_{t_k}, \Delta x_k = x_{k\Delta,(k+1)\Delta}, \Delta \X_k = \X_{k\Delta,(k+1)\Delta}$.        
 We also impose another condition: there exist constants $C(f,V)>0,\mu \geq 1$ such that
\begin{equation}\label{fgrowth}
    \|f(y)\| \leq C(f,V) \Big(1+V(y)^\mu\Big),\quad \forall y \in \R^d.
\end{equation}

We collect all the assumptions into the following one.\\

(${\textbf H}_A$): The system \eqref{tamedEulerDelta} is considered under assumptions (${\textbf H}_{V}$) with $\alpha^{-1},\beta \in \mathcal{K}^{\rm poly}_\infty$ and condition \eqref{gradientgeneral} for $\gamma$ in \eqref{deltamain}, (${\textbf H}_{g}$), (${\textbf H}_{X}$), and conditions \eqref{Mest}, \eqref{fgrowth}.\\

We prove the following auxiliary results. 
\begin{proposition}\label{case1neg}
 Under assumption (${\textbf H}_A$), assume that $s, t \in \Pi^\Delta, s < t$ are two discrete times satisfying \eqref{lambdathreshold}. Then the following estimates hold for the discrete scheme \eqref{tamedEulerDelta} 
\begin{equation}\label{Vestneg1}
    V(y_t) -\bar{C}\leq 4L_V \lambda +\begin{cases}\big[V(y_s)-\bar{C}\big] \big(1- \delta \Delta\big)^{\frac{t-s}{\Delta}}  & \text{if}\quad V(y_s)\leq\bar{C};\\
    \big[ V(y_s) - \bar{C}\big] \Big(1- \frac{\delta \Delta}{1+ M C_{f,V,\mu} \big(1+V(y_s)^\mu\big)\Delta} \Big)^{\frac{t-s}{\Delta}} & \text{if}\quad V(y_s)>\bar{C};\end{cases}
\end{equation}    
for a generic constant $ C_{f,V,\mu}$.
\end{proposition}
\begin{proof}
    Similar to Proposition \ref{case1}, we use the Doss-Sussmann transformation $y_k = \varphi(t_k,s,\bx)z_k$ for $t_k \in \Pi[s,t]$ in \eqref{DS1}, to derive the discrete system for $z_k$ as in \eqref{transformed1} and \eqref{transformed2}, where $\eta, \psi$ satisfy conditions \eqref{etaest} and \eqref{psiest}. As a result, estimate \eqref{Vest1} holds. Under condition \eqref{gradientgeneral} for $\gamma (u) = C_{\lambda_0}-\delta u$, estimate \eqref{Vest2} has the form
    \begin{eqnarray*}
         V(z_{k+1}) &\leq& V(z_k) + \frac{\Delta}{1+ M\|f(z_k+\eta_k)\| \Delta} \Big(C_{\lambda_0} -\delta V(z_k+ \theta^* \Delta z_k)\Big)\\
         &\leq& V(z_k) + \frac{\big[C_{\lambda_0} + L_V \delta \lambda - \delta V(z_k)\big] \Delta}{1+ M\|f(z_k+\eta_k)\| \Delta},\quad \forall s\leq t_k <t.
    \end{eqnarray*}
    As a result,
    \begin{equation}\label{Vestneg2}
        V(z_{k+1}) - \bar{C}\leq \big[ V(z_k) - \bar{C}\big] \Big(1- \frac{\delta \Delta}{1+ M\|f(z_k+\eta_k)\| \Delta} \Big),\quad \forall s\leq t_k <t.
    \end{equation}
    Fix the step size $\Delta$ such that $1 > \delta \Delta$ and write $\bar{C} =\frac{C_{\lambda_0}}{\delta} + L_V \lambda $. Consider three cases.
    \begin{itemize}
        \item If $V(z_s) \leq \bar{C}$, then \eqref{Vestneg2} implies that $V(z_k) \leq \bar{C}$ by induction, and moreover
    \begin{equation}\label{Vestneg3}
        V(z_k) - \bar{C}\leq \big[ V(z_s) - \bar{C}\big] \big(1- \delta \Delta\big)^{\frac{t_k-s}{\Delta}},\quad \forall t_k \in \Pi^\Delta[s,t].
    \end{equation}
        \item If $V(z_k) \leq \bar{C} < V(z_s)$ for some $t_k \in \Pi^\Delta[s,t]$ then the left hand side of \eqref{Vestneg3} is non-positive while the right hand side is positive, thus \eqref{Vestneg3} still holds.
        \item If $\bar{C} < V(z_k)$ for some $t_k \in \Pi[s,t]$ then by \eqref{Vestneg2}, $\bar{C} < V(z_i)$ for all $t_i \in \Pi^\Delta[s,t_k]$ and the sequence $\{V(z_i): t_i \in \Pi^\Delta[s,t_k]\}$ is strictly decreasing. On the other hand, it follows from \eqref{fgrowth} that
        \begin{equation}\label{CfVmu}
        \|f(z_k+\eta_k)\| \leq C_{f,V}\big(1+V(z_k+\eta_k)^\mu\big)\leq C_{f,V}\Big(1+\big[V(z_k)+4L_V\lambda)\big]^\mu\Big) \leq C_{f,V,\mu} \big(1+V(z_k)^\mu\big)
         \end{equation}
        for a generic constant $C_{f,V,\mu}$. As a result, \eqref{Vestneg2} yields
    \begin{eqnarray*}
       0 < V(z_{i+1}) - \bar{C} &\leq& \big[ V(z_i) - \bar{C}\big] \Big(1- \frac{\delta \Delta}{1+ M C_{f,V,\mu} \big(1+V(z_i)^\mu\big)\Delta} \Big)\\
        &\leq& \big[ V(z_i) - \bar{C}\big] \Big(1- \frac{\delta \Delta}{1+ M C_{f,V,\mu} \big(1+V(z_s)^\mu\big)\Delta} \Big),\quad \forall s\leq t_i <t_k.
    \end{eqnarray*}
        Hence by induction, it follows in this case that
        \begin{equation}\label{Vestneg4}
            0 < V(z_k) - \bar{C} \leq \big[ V(z_s) - \bar{C}\big] \Big(1- \frac{\delta \Delta}{1+ M C_{f,V,\mu} \big(1+V(z_s)^\mu\big)\Delta} \Big)^{\frac{t_k-s}{\Delta}}.
        \end{equation}
    \end{itemize}
    By assigning $t_k = t$ and using the fact that $V(y_t) \leq V(z_t) + 4L_V \lambda,\ V(y_s) = V(z_s)$, we obtain \eqref{Vestneg1}.
\end{proof}
\begin{proposition}\label{case2neg}
    Under assumption (${\textbf H}_A$), assume $s,t \in \Pi^\Delta, s <t$ are two consecutive times in $\Pi^\Delta$. Then the following estimates hold for a generic constant $C_{f,V,\mu}$
\begin{eqnarray}\label{Vestneg5}
     V(y_t) -\bar{C}&\leq& L_V\Big( C_g \| x_{s,t}\| + C_g^2 \| \X_{s,t}\|\Big)\notag\\
     && +\begin{cases}\big[V(y_s)-\bar{C}\big] \Big(1- \delta \Delta\Big)  & \text{if}\quad V(y_s)\leq\bar{C};\\
    \big[ V(y_s) - \bar{C}\big] \Big(1- \frac{\delta \Delta}{1+ M C_{f,V,\mu} \big(1+V(y_s)^\mu\big)\Delta} \Big)& \text{if}\quad V(y_s)>\bar{C}.\end{cases}
\end{eqnarray}
\end{proposition}
\begin{proof}
    The proof is straightforward from \eqref{lipschitzV} and the Lagrange mean value theorem. Indeed, 
 \allowdisplaybreaks
    \begin{eqnarray*}
        V(y_t) &\leq& V\Big(y_s + \frac{f(y_s) \Delta}{1+ M\|f(y_s)\|\Delta}\Big) + L_V\Big(C_g \| x_{s,t}\| + C_g^2 \| \X_{s,t}\|\Big)\\
        &\leq& V(y_s) +\Big\langle \nabla V\Big(y_s + \theta^*\frac{f(y_s) \Delta}{1+ M\|f(y_s)\|\Delta} \Big),\frac{f(y_s)}{1+ M\|f(y_s)\|\Delta}\Big\rangle\Delta+ L_V\Big(C_g \| x_{s,t}\| + C_g^2 \| \X_{s,t}\|\Big)\\
        &\leq& V(y_s) + \Big[C_{\lambda_0} - \delta V\Big(y_s + \theta^*\frac{f(y_s) \Delta}{1+ M\|f(y_s)\|\Delta} \Big)  \Big]\frac{\Delta}{1+ M\|f(y_s)\|\Delta} + L_V\Big( C_g \| x_{s,t}\| + C_g^2 \| \X_{s,t}\|\Big)\\
        &\leq& V(y_s) + \frac{\big[\bar{C}- V(y_s)\big]\delta\Delta}{1+ M\|f(y_s)\|\Delta} 
        + L_V\Big( C_g \| x_{s,t}\| + C_g^2 \| \X_{s,t}\|\Big)
    \end{eqnarray*}
    which proves 
    \[
        V(y_t) - \bar{C}\leq \big[ V(y_s) - \bar{C}\big] \Big(1- \frac{\delta \Delta}{1+ M\|f(y_s)\| \Delta} \Big)+ L_V\Big( C_g \| x_{s,t}\| + C_g^2 \| \X_{s,t}\|\Big).
    \]
    The rest of the proof follows similar arguments and estimates to the proof of Proposition \ref{case1neg}.
\end{proof}
\begin{theorem}\label{tamedschemeneg}
    Under assumption (${\textbf H}_A$), the solution $y$ of the tamed numerical scheme \eqref{tamedEulerDelta} satisfies for all $t\in \Pi$ the estimate
    \begin{equation}\label{Vestnegmain}
        \Big|V(y_t) -\bar{C}\Big|
        \leq  \Big|V(y_0) -\bar{C}\Big| \exp \Bigg\{-\frac{\delta t}{1+\Delta MC_{f,V,\mu}\Big(1+ V(y_0)^\mu + H(\bx,\Pi^\Delta[0,t])^\mu\Big)} \Bigg\} + H(\bx,\Pi^\Delta[0,t]);
    \end{equation}    
 where
 \begin{equation}\label{Hbxneg}
 H\big(\bx,\Pi^\Delta[0,t]\big)= \max \Big\{\bar{C}, L_V\Big(4\lambda + C_g \|\Delta \bx\|_{\infty,\Pi^\Delta[0,t]} + C_g^2 \|\Delta \bx\|_{\infty,\Pi^\Delta[0,t]}^2\Big)N\Big(\frac{\lambda}{C_pC_g},\bx,\Pi^\Delta[0,t]\Big) \Big\}.
    \end{equation}
\end{theorem}
\begin{proof}
  A direct consequence of the estimates \eqref{Vestneg1} in Propositions \ref{case1neg} and \eqref{Vestneg5} in Proposition \ref{case2neg} shows that: for a sequence of stopping times $\Big\{\tau^\Delta_n(\frac{\lambda}{C_pC_g},\bx,\Pi^\Delta[0,t])\Big\}$, we obtain the general estimate 
  \begin{eqnarray}\label{Vestneg6}
  V(y_{\tau^\Delta_{n+1}}) -\bar{C} &\leq& L_V\Big(4\lambda + C_g \|\Delta \bx\|_{\infty,\Pi[0,t]} + C_g^2 \|\Delta \bx\|_{\infty,\Pi[0,t]}^2\Big)\\
  &&+ \begin{cases}\big[V(y_{\tau^\Delta_n}) -\bar{C}\big] (1-\delta \Delta)^{\frac{\tau^\Delta_{n+1}-\tau^\Delta_n}{\Delta}} & \text{if\ } V(y_{\tau^\Delta_n}) \leq \bar{C} \\ \big[V(y_{\tau^\Delta_n}) -\bar{C}\big] \Big(1-\frac{\delta \Delta}{1+\Delta MC_{f,V,\mu}\big(1+V(y_{\tau^\Delta_n})^\mu\big)}\Big)^{\frac{\tau^\Delta_{n+1}-\tau^\Delta_n}{\Delta}} & \text{if\ } V(y_{\tau^\Delta_n}) > \bar{C} \end{cases}      \notag
  \end{eqnarray}
  for any two consecutive stopping times in $\Big\{\tau^\Pi_n(\frac{\lambda}{C_pC_g},\bx,\Pi^\Delta[0,t])\Big\}$. Using \eqref{Vestneg6} and induction arguments, it is easy to prove that for any $0\leq n \leq N\Big(\frac{\lambda}{C_pC_g},\bx,\Pi^\Delta[0,t]\Big)$
  \begin{eqnarray}\label{Vestneg7}
  V(y_{\tau^\Delta_n}) -\bar{C} &\leq& n L_V\Big(4\lambda + C_g \|\Delta \bx\|_{\infty,\Pi^\Delta[0,t]} + C_g^2 \|\Delta \bx\|_{\infty,\Pi^\Delta[0,t]}^2\Big)\\
  &&+ \begin{cases}\big[V(y_{\tau^\Delta_0}) -\bar{C}\big] (1-\delta \Delta)^{\frac{\tau^\Delta_n-\tau^\Delta_0}{\Delta}} & \text{if\ } V(y_{\tau^\Delta_0}) \leq \bar{C} \\ \big[V(y_{\tau^\Delta_0}) -\bar{C}\big] \Big(1-\frac{\delta \Delta}{1+\Delta MC_{f,V,\mu}\big(1+\max \limits_{0\leq i \leq n}V(y_{\tau^\Delta_i})^\mu\big)}\Big)^{\frac{\tau^\Delta_n-\tau^\Delta_0}{\Delta}} & \text{if\ } V(y_{\tau^\Delta_0}) > \bar{C}. \end{cases}      \notag
  \end{eqnarray}
A direct consequence of \eqref{Vestneg7} is that 
\[
  V(y_{\tau^\Delta_n}) \leq \bar{C} + V(y_{\tau^\Delta_0}) + n L_V\Big(4\lambda + C_g \|\Delta \bx\|_{\infty,\Pi^\Delta[0,t]} + C_g^2 \|\Delta \bx\|_{\infty,\Pi^\Delta[0,t]}^2\Big).
\]
Hence 
\[
\max  \limits_{0\leq i \leq n}  V(y_{\tau^\Delta_n}) \leq \bar{C} + V(y_{\tau^\Delta_0}) + H(\bx,\Pi^\Delta[0,t]),\quad \forall 0\leq n \leq N\Big(\frac{\lambda}{C_pC_g},\bx,\Pi^\Delta[0,t]\Big),
\]
which, together with \eqref{Vestneg7}, yields
 \allowdisplaybreaks
  \begin{eqnarray}\label{Vestneg8}
  &&\Big|V(y_{\tau^\Delta_n}) -\bar{C}\Big| \leq H(\bx,\Pi^\Delta[0,t])+ \Big|V(y_{\tau^\Delta_0}) -\bar{C}\Big| \times \notag\\
  && \times \max\Big\{(1-\delta \Delta)^{\frac{\tau^\Delta_n-\tau^\Delta_0}{\Delta}},\Big(1-\frac{\delta \Delta}{1+\Delta  MC_{f,V,\mu}\big(1+ \big[\bar{C} + V(y_{\tau^\Delta_0}) + H(\bx,\Pi^\Delta[0,t])\big]^\mu\big)}\Big)^{\frac{\tau^\Delta_n-\tau^\Delta_0}{\Delta}} \Big\}  \notag\\
  &\leq& H(\bx,\Pi^\Delta[0,t])+ \Big|V(y_{\tau^\Delta_0}) -\bar{C}\Big| \Big(1-\frac{\delta \Delta}{1+\Delta  MC_{f,V,\mu}\big(1+ \big[\bar{C} + V(y_{\tau^\Delta_0}) + H(\bx,\Pi^\Delta[0,t])\big]^\mu\big)}\Big)^{\frac{\tau^\Delta_n-\tau^\Delta_0}{\Delta}} \notag\\
  &\leq& H(\bx,\Pi^\Delta[0,t])+ \Big|V(y_{\tau^\Delta_0}) -\bar{C}\Big| \exp \Big\{-\frac{\delta (\tau^\Delta_n-\tau^\Delta_0)}{1+\Delta  MC_{f,V,\mu}\big(1+ V(y_{\tau^\Delta_0})^\mu + H(\bx,\Pi^\Delta[0,t])^\mu\big)} \Big\}
  \end{eqnarray}
for a generic constant $C_{f,V,\mu}$. In particular, by replacing $n:=N\Big(\frac{\lambda}{C_pC_g},\bx,\Pi^\Delta[0,t]\Big)$ in \eqref{Vestneg8} and using $\tau^\Delta_0 =0, \tau^\Delta_N =t$, we obtain \eqref{Vestnegmain}.

\end{proof}

\begin{corollary}\label{corV}
    Under assumption (${\textbf H}_A$), there exist generic constants $\delta_{f,V,\mu}, C_{f,V,\mu}$ and a generic function $\bar{H}(\Delta,\bx,[0,1])$ such that solution $y$ of the tamed numerical scheme \eqref{tamedEulerDelta} satisfies 
    \begin{equation}\label{Vestnegmain1}
        \Big|V(y_1) -\bar{C}\Big|^\mu
        \leq  \Big|V(y_0) -\bar{C}\Big|^\mu \exp \Bigg\{\frac{-\delta_{f,V,\mu}}{1 +\Delta MC_{f,V,\mu}\Big(\big|V(y_0)-\bar{C}\big|^\mu + \bar{H}(\bx,[0,1])\Big)} \Bigg\} + \bar{H}(\bx,[0,1]).
    \end{equation}    
\end{corollary}
\begin{proof}
    The proof follows directly from \eqref{Vestnegmain} for $t=1$ by applying the inequality 
     \allowdisplaybreaks
    \begin{eqnarray*}
    &&\frac{-\delta }{1+\Delta MC_{f,V,\mu}\Big(1+ V(y_0)^\mu + H(\bx,\Pi^\Delta[0,1])^\mu\Big)} \\
    &=&\frac{-\frac{\delta}{1+\Delta MC_{f,V,\mu}}}{1+\frac{\Delta MC_{f,V,\mu}}{1+\Delta MC_{f,V,\mu}}\Big(V(y_0)^\mu + H(\bx,\Pi^\Delta[0,1])^\mu\Big)} \\
    &\leq&\frac{-\delta(1-\Delta MC_{f,V,\mu})}{1+\Delta M\frac{C_{f,V,\mu}}{1+\Delta MC_{f,V,\mu}}\Big(V(y_0)^\mu + H(\bx,\Pi^\Delta[0,1])^\mu\Big)} \\
    &\leq&\frac{-\delta(1-\Delta MC_{f,V,\mu})}{1+\Delta MC_{f,V,\mu}\Big(V(y_0)^\mu + H(\bx,\Pi^\Delta[0,1])^\mu\Big)} 
    \end{eqnarray*}
    for sufficiently small $\Delta$ such that $\Delta MC_{f,V,\mu} < 1$. Then using Jensen inequality of the form 
    \[
    (a+b)^\mu \leq  b^\mu \Big(\frac{1+\kappa}{\kappa}\Big)^{\mu-1}+a^\mu (1+\kappa)^{\mu-1}  \leq  b^\mu \Big(\frac{1+\kappa}{\kappa}\Big)^{\mu-1}+a^\mu e^{\kappa(\mu-1)} ,\quad \forall a,b,\kappa >0
    \]
    yields
    \begin{eqnarray}\label{Vestneg9}
        \Big|V(y_1) -\bar{C}\Big|^\mu
        &\leq&  \Big(\frac{1+\kappa}{\kappa}\Big)^{\mu-1}H(\bx,\Pi^\Delta[0,1])^\mu + \Big|V(y_0) -\bar{C}\Big|^\mu  \notag\\
        && \times \exp \Bigg\{\frac{-\mu \delta (1- \Delta MC_{f,V,\mu}) + \kappa(\mu-1)}{1+\Delta MC_{f,V,\mu}\Big(\big|V(y_0)-\bar{C}\big|^\mu + H(\bx,\Pi^\Delta[0,1])^\mu\Big)} \Bigg\}\notag\\
        &\leq& \Big(\frac{1+\kappa}{\kappa}\Big)^{\mu-1}H(\bx,\Pi^\Delta[0,1])^\mu + \Big|V(y_0) -\bar{C}\Big|^\mu  \\
        && \times \exp \Bigg\{\frac{-\mu \delta (1- \Delta MC_{f,V,\mu}) + \kappa(\mu-1)}{1+\Delta MC_{f,V,\mu}\Big(\big|V(y_0)-\bar{C}\big|^\mu + \Big(\frac{1+\kappa}{\kappa}\Big)^{\mu-1}H(\bx,\Pi^\Delta[0,1])^\mu\Big)} \Bigg\}.\notag
    \end{eqnarray}    
    Obverse that, due to \eqref{Hbxneg}, \eqref{Nest} and the estimate
    \[
    \|\Delta \bx\|_{\infty,\Pi^\Delta[0,1]} \leq \|\Delta \bx\|_{\infty,\Delta,[0,1]}\leq \ltn \bx \rtn_{\tp,[0,1]}, 
    \]
    $\Big(\frac{1+\kappa}{\kappa}\Big)^{\mu-1}H$ is then bounded from above by
    \begin{eqnarray}\label{barH}
       &&\Big(\frac{1+\kappa}{\kappa}\Big)^{\mu-1} H(\bx,\Pi^\Delta,[0,1])\notag\\
       &\leq& \Big(\frac{1+\kappa}{\kappa}\Big)^{\mu-1} \max \Big\{\bar{C}, 2L_V\Big(4\lambda + C_g \ltn\bx\rtn_{\tp,[0,1]} + C_g^2 \ltn\bx\rtn_{\tp,[0,1]}^2\Big)\times\notag\\
       &&\hspace{7cm}\times \Big(1+\frac{C_p^pC_g^p}{\lambda^p}\ltn\bx\rtn^p_{\tp,[0,1]}\Big) \Big\} \notag\\
       &\leq& C(\kappa,\bar{C},\lambda,C_g) \Big(1+\ltn\bx\rtn^{p+2}_{\tp,[0,1]}\Big)=: \bar{H}(\bx,[0,1])
    \end{eqnarray}
Therefore, as an increasing function of $\Big(\frac{1+\kappa}{\kappa}\Big)^{\mu-1}H$, the right hand side of \eqref{Vestneg9} is less than the right hand side of \eqref{Vestnegmain1} for generic constants $\delta_{f,V,\mu}, C_{f,V,\mu}$ and a generic function $\bar{H}(\bx,[0,1])$ (dependent on $\kappa$) defined by \eqref{barH}, where we choose $\Delta,\kappa >0$ sufficiently small such that 
    \begin{equation}\label{deltafVmu}
    1-\Delta MC_{f,V,\mu} >0,\quad \delta_{f,V,\mu}:=\mu \delta \big(1- \Delta MC_{f,V,\mu}\big) - \kappa(\mu-1) >0.
    \end{equation}
\end{proof}

We need the following auxiliary results (see the proofs in Appendix \ref{proofs}) 
\begin{lemma}\label{xpvarpoly}
i, Under assumption (${\textbf H}_X$), for any constants $\rho \geq 1,r_0>1$, there exists a positive random variable $C$, $C(\omega) \in (0,\infty)$ a.s. such that       
\begin{equation}\label{xpvarpoly1}
    \ltn \bx(\omega) \rtn^\rho_{\tp,[-n-1,-n]} \leq C(\omega) (1+n^{r_0}),    \quad \forall n \in \N, \forall \omega \in \Omega \quad \text{a.s.}
\end{equation}
ii, If $X$ is a Gaussian rough path, then there exist constants $r_0 >1, C_0>0$ such that 
\begin{equation}\label{Ebx}
     \Big(\E \ltn \bx(\omega) \rtn_{\tp,[-1,0]}^{\rho n}\Big)^{\frac{1}{n}} \leq C_0\big(1+n^{r_0}\big),\quad \forall n \in \N. 
\end{equation}
\end{lemma}

\begin{lemma}\label{lemmaF1}
    Assume $\xi(\cdot) \in L^1$ is a random variable and there exists a positive random variable $C(\xi,\omega)$ and constants $C_0(\xi),r(\xi)>0$ such that   
    \begin{equation}\label{xipoly}
    \xi(\theta_{-n} \omega) \leq C(\xi,\omega) \big[1+n^{r(\xi)}\big],\quad \forall n \in \N, \forall \omega \in \Omega.
    \end{equation}
Let $\delta,\epsilon$ be constants satisfying  
    \begin{equation}\label{deltaeps}
2>\delta >  \epsilon \big[1+ r(\xi)\big] \E \xi. 
     \end{equation}
    Then, given a random variable $d(\cdot)<\infty$ a.s., the random series
    \begin{equation}\label{Rlem}
R\Big(\epsilon,d(\omega),\xi(\omega)\Big):= \sum_{n=0}^{\infty} \xi(\theta_{-n} \omega) \exp \Big\{\sum_{j =0}^{n-1}\frac{-\delta}{1+ \epsilon d(\omega) + \epsilon \sum_{i = 0}^j \xi(\theta_{-i} \omega)}\Big\}
\end{equation}
is finite a.s. If in addition $d(\cdot) \in L^1$ and 
\begin{equation}\label{Exi}
    \big(\E \xi^n\big)^{\frac{1}{n}} \leq C_0(\xi) \Big[1+n^{r(\xi)}\Big],\quad \forall n \in \N,
\end{equation}
then $R(\epsilon,d,\xi) \in L^1$.
\end{lemma}

\begin{remark}\label{Rrem}
  i, The definition of $R\Big(\epsilon,d(\omega),\xi(\omega)\Big)$ in \eqref{Rlem} shows that $R$ is increasing w.r.t. $\epsilon$. Moreover, for a.s. $\omega \in \Omega$, the decomposition \eqref{Rest} shows that $R\Big(\epsilon,d(\omega),\xi(\omega)\Big)$ is continuous w.r.t. $\epsilon$. In particular,
  \begin{equation}\label{R0}
  \lim \limits_{\epsilon \downarrow 0} R\Big(\epsilon,d(\omega),\xi(\omega)\Big) = R\Big(0,d(\omega),\xi(\omega)\Big) =\sum_{n=0}^\infty \xi(\theta_{-n}\omega) e^{-n\delta}
  =: R(0,\xi(\omega))\quad \text{a.s.}
  \end{equation}
  where $R(0,\xi)$ is independent of $\|u_0\|_\infty$.\\

  ii,  In the deterministic case of $\xi$ and $d_0$, we have the estimate
    \begin{equation}\label{Rlem2}
R(\epsilon,u_0,\xi):= \xi \sum_{n=0}^{\infty} \exp \Big\{\sum_{j =0}^{n-1}\frac{-\delta}{1+ \epsilon d_0+ \epsilon j\xi }\Big\} \approx \xi \sum_{n=1}^\infty n^{-\frac{\delta}{\epsilon \xi}} = \xi \zeta\Big(\frac{\delta}{\epsilon \xi}\Big).
\end{equation}

\end{remark}

\begin{lemma}\label{lemFprop}
    Denote by $F$ the two variable function
    \begin{equation}\label{F2}
        F(\xi,u) = u \exp \Big\{-\frac{\delta }{1+\epsilon u+ \epsilon \xi}\Big\}  + \xi,\quad \forall u,\xi \geq 0.
    \end{equation}
    Then $F$ is a strictly increasing function of $u$ and $\epsilon$. In particular,
    \begin{equation}\label{Fprop}
        0< \frac{\partial F}{\partial u} (\xi,u) < 1- \frac{\delta(1+\delta+\epsilon \xi)}{(1+\delta+\epsilon \xi +\epsilon u)^2}.
    \end{equation}
\end{lemma}

\begin{lemma}\label{lemmaF2}
    Consider a sequence of random variables $\{u_k\}_{k\in \N}$ such that 
    \begin{equation}\label{u0}
    \|u_0(\omega)\|_\infty:=\sup \limits_{n \in \N} u_0(\theta_{-n}\omega) < \infty\quad \text{a.s.,}\quad \E\|u_0(\omega)\|_\infty <\infty,
    \end{equation}
    and 
    \begin{equation}\label{uk1F}
        u_{k+1}(\omega) = F(\xi(\theta_{k+1}\omega),u_k(\omega)),\quad \forall k \in \N, \omega \in \Omega.
    \end{equation}
Then under the assumptions of Lemma \ref{lemmaF1}, 
    \begin{equation}\label{uk2F}
    u_n(\theta_{-n}\omega) \leq \|u_0(\omega)\|_\infty + R\Big(\epsilon,\|u_0(\omega)\|_\infty,\xi(\omega)\Big) <\infty,\quad \forall n \in \N\quad \text{a.s.}
    \end{equation}
Moreover, there exists a unique random pullback omega limit set $u_\infty(\omega) = u_\infty(\delta,\epsilon,\xi,\omega)$ given by
\begin{equation}\label{pullbacku}
    u_\infty(\omega) = \bigcap_{n=1}^\infty \overline{\bigcup_{k = n}^\infty u_k(\theta_{-k}\omega)} \quad \text{a.s.} 
\end{equation}
which is compact and independent of $u_0$, such that 
\begin{equation}\label{uinfty}
\limsup \limits_{n \to \infty} u_n(\theta_{-n}\omega) \leq |u_\infty(\omega)| \leq R\Big(\epsilon,\|u_0(\omega)\|_\infty,\xi(\omega)\Big) <\infty,\quad \text{a.s.}    
\end{equation}
If in addition \eqref{Exi} holds, then $|u_\infty(\cdot)|:= \max\{r: r\in u_\infty(\cdot)\} \in L^1$. 
\end{lemma}

\begin{remark}\label{remu}
  i,  It follows from the ergodicity of the metric dynamical system $\theta$ that condition \eqref{u0} is equivalent to: there exists an subset $\Omega^0 \subseteq \Omega$ of full measure such that 
    \begin{equation}\label{u0bounded}
       \|u_0\|_{\infty,\Omega^0}:= \sup_{\omega \in \Omega^0} u_0(\omega) < \infty,\quad \forall \omega \in \Omega^0.
    \end{equation}
    Indeed, since the sets $\Omega^{0,\infty}_{k}:=\{\omega: \|u_0(\omega)\|_\infty \leq k\}$ are measurable, there exists $\mP(\Omega^{0,\infty}_{k_0}) >0$ for $k_0 \in \N$ large enough. Observe that for any $\omega_0 \in \Omega^{0,\infty}_{k_0}$ then $\theta_{-n}\omega_0 \in \Omega^0_{k_0} :=\{\omega: u_0(\omega) \leq {k_0}\}$ by definition, thus $\omega_0 \in \bigcap_{n=0}^\infty\theta_n \Omega^0_{k_0}$. This proves $\Omega^{0,\infty}_{k_0}\subseteq \bigcap_{n=0}^\infty\theta_n \Omega^0_{k_0}$. The ergodicity of $\theta$ yields
    \[
    0<\mP(\Omega^{0,\infty}_{k_0})  = \lim \limits_{n \to \infty} \frac{1}{n} \sum_{i =1}^n \mP(\theta_n \Omega^0_{k_0} \cap \Omega^{0,\infty}_{k_0}) = \mP(\Omega^{0,\infty}_{k_0}) \mP(\Omega^0_{k_0}).
    \]
    This happens only if $\mP(\Omega^0_{k_0}) =1$ or $\|u_0\|_{\infty,\Omega^0_{k_0}} \leq k_0 <\infty$. From now on we will use \eqref{u0bounded} in replace of \eqref{u0}.

    ii, By Lemma \ref{lemFprop} and the comparison principle, $u_n(\theta_{-n}\omega)$ is also increasing in $\epsilon$, thus $|u_\infty(\delta,\epsilon,\xi,\omega)|$ is non-decreasing in $\epsilon$. In particular, a direct computation shows that $|u_\infty(\delta,0,\xi,\omega)| = R(0,\xi(\omega))$. Due to \eqref{uinfty} and \eqref{R0},
    \begin{equation}\label{uinfty0}
\lim \limits_{\epsilon \to 0}  |u_\infty(\delta,\epsilon,\xi,\omega)| = R(0,\xi(\omega))=|u_\infty(\delta,0,\xi,\omega)|.
    \end{equation}
\end{remark}

We are now in the position to state our main theorem.
\begin{theorem}\label{tamedattractor}
    Under assumption (${\textbf H}_A$), there exists a $\Delta_0 >0$ sufficiently small such that for any step size $0<\Delta <\Delta_0$, the discrete random dynamical system $\Phi_{C_g}^\Delta$ of the tamed numerical scheme \eqref{tamedEulerDelta} with the regular grid $\Pi^\Delta$ admits a global numerical attractor $\cA_{C_g}^\Delta(\omega)$ which is upper semi-continuous w.r.t. $\Delta$ and $C_g$ in the sense that
    \begin{equation}\label{numattractorrough}
    \begin{split}
       & \lim \limits_{\Delta \to 0} d_H(\cA_{C_g}^\Delta(\cdot)|\cA^0_{C_g}(\cdot)) = 0 \quad \text{a.s.}\\
       & \lim \limits_{C_g \to 0} d_H(\cA_{C_g}^\Delta(\cdot)|\cA^\Delta_0(\cdot)) = 0 \quad \text{a.s.}
    \end{split}
    \end{equation}
    Moreover, if $X$ is a Gaussian rough path, then $\cA^\Delta \in L^1$ and the convergence in \eqref{numattractorrough} holds also in the $L^1$ sense.
\end{theorem}
\begin{proof}
    As a direct consequence of Theorem \ref{tamedschemeneg} and Corollary \ref{corV}, it follows that $V(y_1)=V(\Phi^\Delta(1,\omega)y_0)$ and $V(y_0)$ satisfies \eqref{Vestnegmain1} for $\bar{H}$ given by \eqref{barH}. By comparison principle, 
    \[
    |V(\Phi^\Delta(k,\omega)y_0)-\bar{C}|^\mu \leq u_k(\omega),\quad \forall k \in \N,\omega \in \Omega 
    \]
    where $u_k$ comes from Lemma \ref{lemmaF2} for 
    \[
    u_0(\omega):= |V(y_0)-\bar{C}|^\mu, \quad \xi(\omega) := \bar{H}(\bx(\omega),[0,1])
    \]
    and 
    \[
    \epsilon:= \Delta MC_{f,V,\mu},\quad  \delta := \delta_{f,V,\mu}
    \]
    in \eqref{deltafVmu}. Here we choose $r(\xi):=r_0 >1$ a fixed constant in \eqref{xipoly} (this is valid due to \eqref{xpvarpoly1} in Lemma \ref{xpvarpoly}) and adapt condition \eqref{deltaeps} to the form
    \begin{equation}\label{deltaepsMain}
\delta_{f,V,\mu}=\mu \delta \big(1- \Delta MC_{f,V,\mu}\big) - \kappa(\mu-1) >  \Delta MC_{f,V,\mu} (1+ r_0) \E \bar{H}(\bx,[0,1]) 
     \end{equation}
 for $\delta$ in \eqref{deltamain}, $\mu$ in \eqref{fgrowth}, $M$ in \eqref{tamedEulerDelta} and $C_{f,V,\mu}$ in \eqref{CfVmu}. Condition \eqref{deltaepsMain} is equivalent to
 \begin{equation}
 \Delta_0:=\frac{\mu \delta - \kappa (\mu-1)}{MC_{f,V,\mu} \Big[(1+ r_0) \E \bar{H}(\bx,[0,1]) +\mu \delta\Big]} > \Delta.     
 \end{equation}
 Therefore, given any deterministic bounded set $D\subset \R^d$ in the universe $\cD$ given by
 \begin{equation}\label{univD}
\cD:=  \Big\{D \subset \R^d: D \quad \text{is a bounded set} \Big\},
 \end{equation} 
 we can prove by \eqref{uk2F} that for any $y_0 \in D$
 \[
 |V(\Phi^\Delta(n,\theta_{-n}\omega)y_0)-\bar{C}|^\mu \leq |V(y_0)-\bar{C}|^\mu+ R\Big(\Delta MC_{f,V,\mu},|V(y_0)-\bar{C}|^\mu,\bar{H}(\bx(\omega),[0,1])\Big),\quad \forall n\in \N.
 \]
Moreover by \eqref{uinfty}, 
\[
\limsup \limits_{n \to \infty}  |V(\Phi^\Delta(n,\theta_{-n}\omega)y_0)-\bar{C}|^\mu \leq \big|u_\infty(\delta_{f,V,\mu}\omega,\Delta MC_{f,V,\mu},\bar{H},\omega)\big|,\quad \forall y_0\in D;
\]
hence there exists an absorbing set 
 \[
 \cB_{C_g}^\Delta(\omega) = V^{-1}\Big(\bar{C}+\Big[1+\big|u_\infty(\delta_{f,V,\mu}\omega,\Delta MC_{f,V,\mu},\bar{H},\omega)\big|\Big]^{\frac{1}{\mu}}\Big)
 \]
 such that $\Phi^\Delta(n,\theta_{-n}\omega)D$ is absorbed in the pullback sense to $\cB_{C_g}^\Delta(\omega)$ a.s. 
 Hence, there exists a numerical random pullback attractor $\cA_{C_g}^\Delta$ for $\Phi_{C_g}^\Delta$ given by 
 \begin{equation}\label{ABgen}
 \cA_{C_g}^\Delta(\omega) = \bigcap_{n=1}^\infty \overline{\bigcup_{k = n}^\infty \Phi_{C_g}^\Delta(k,\theta_{-k}\omega) \cB_{C_g}^\Delta(\theta_{-k}\omega)}
 \end{equation}
 such that $\cA_{C_g}^\Delta$ attracts  every deterministic bounded set in the universe $\cD$ a.s. in the pullback sense. Moreover, due to Remark \ref{remu} (ii), $\big|u_\infty(\delta_{f,V,\mu}\omega,\Delta MC_{f,V,\mu},\bar{H},\omega)\big|$ is non-decreasing and satisfies \eqref{uinfty0}, thus $\cB_{C_g}^\Delta$ is non-decreasing in $\Delta$ w.r.t. inclusion and $d_H(\cB_{C_g}^\Delta(\omega) \big| \cB_{C_g}^0(\omega))  \downarrow 0$ as $\Delta \downarrow 0$ a.s. where $\cB_{C_g}^0(\omega) = V^{-1}\Big(\bar{C}+\Big[1+R\Big(0,\bar{H}(\bx(\omega),[0,1])\Big)\Big]^{\frac{1}{\mu}}\Big)$ is the pullback absorbing set of the continuous RDS $\Phi_{C_g}$ which is used to generate $\cA_{C_g}^0$ using \eqref{ABgen}. We can therefore apply similar arguments in \cite[Theorem 5.3]{congduchong23} to prove that $\cA_{C_g}^\Delta$ is upper semi-continuous w.r.t. $\Delta$ and \eqref{numattractorrough} holds a.s. \\

 To prove the upper semi-continuity of $\cA_{C_g}^\Delta$ w.r.t. $C_g$, observe from \eqref{Cbar} and \eqref{barH} that we can assign $\lambda := C_g$ for $C_g$ sufficiently small such that \eqref{Mest} holds. Then, given $\lambda := C_g$, $\bar{H}$ is increasing in $C_g$, and moreover 
 \[
\lim \limits_{\lambda = C_g \to 0} \bar{H}(\bx,[0,1])=\lim \limits_{ C_g \to 0} C\Big(\kappa,\frac{C_{\lambda_0}}{\delta}+L_VC_g,C_g,C_g\Big) \Big(1+\ltn\bx\rtn^{p+2}_{\tp,[0,1]}\Big) = \Big(\frac{1+\kappa}{\kappa}\Big)^{\mu-1} \frac{C_{\lambda_0}}{\delta} \quad \text{a.s.}
 \]
Hence, $\cB_{C_g}^\Delta$ is increasing in $C_g$ w.r.t. inclusion and $d_H(\cB_{C_g}^\Delta \big| \cB_{0}^\Delta(\omega))  \downarrow 0$ as $C_g \to 0$. The similar arguments as before are applied to prove the second limit in \eqref{numattractorrough}.\\

 Finally, if $X$ is a Gaussian rough path, then due to \eqref{barH} and \eqref{Ebx}, $\xi = \bar{H}$ satisfies \eqref{Exi} for $r(\xi) = p+2$, hence 
 $R\Big(\Delta M C_{f,V,\mu},|V(D)-\bar{C}|^\mu,\bar{H}(\bx(\cdot),[0,1])\Big) \in L^1$ due to Lemma \ref{lemmaF1}. Thus 
 \[
 \big|u_\infty(\delta_{f,V,\mu}\omega,\Delta MC_{f,V,\mu},\bar{H},\omega)\big|,|\cB_{C_g}^\Delta(\cdot)|, |\cA_{C_g}^\Delta(\cdot)| \in L^1.
 \]
 Similar to \cite[Theorem 17]{ducjost25}, the convergence in $L^1$ in the limits \eqref{numattractorrough} follows from  Lebesgue's dominated theorem, the non-decrease in $\Delta$ and $C_g$ (when we assign $\lambda := C_g$), and the integrability of $\cA_{C_g}^\Delta$.\\
\end{proof}

\begin{example}[Pitchfork system under fractional noises]\label{mul.noise}
Consider the scalar SDE driven by a fractional Brownian motion with $H \in (\frac{1}{3},\frac{1}{2}]$
\begin{equation}\label{boundedpitchfork}
    dy = (\alpha y -y^3)dt + C_g\tanh(y) dB^H. 
\end{equation}
where $\alpha, C_g >0$. A direct computation shows that $g(y) =C_g \tanh(y) = C_g\frac{e^y -e^{-y}}{e^y +e^{-y}}\in C_b^\infty(\R,[-1,1])$ is a strictly increasing, $g(0)=0$ and 
\[
\max \Big\{\|g\|_\infty,\|Dg\|_\infty,\|D^2g\|_\infty, \|D^3g\|_\infty \Big\} \leq 2C_g.
\]
Because the drift $f(y) = \alpha y -y^3$ is globally dissipative, the generated semigroup $\Phi_0$ of the unperturbed system admits a global attractor $\cA^0_0 = [-\sqrt{\alpha},\sqrt{\alpha}]$, where $\sqrt{\alpha}$ is a fixed point of $\Phi_0$. On the other hand, it follows from \cite[Section 2.2]{ducjost25} that, by choosing the strong Lyapunov function $V(y) = \sqrt{1+\|y\|^2}$, 
there exists a unique pathwise solution for system \eqref{boundedpitchfork}, of which zero is the trivial solution. Moreover, the generated RDS $\Phi_{C_g}(t,\omega)y_0$ from \eqref{boundedpitchfork} satisfies the oder-preserving property, i.e. $\Phi_{C_g}(t,\omega)y_0 > \Phi_{C_g}(t,\omega)\bar{y}_0$ for any $y_0 > \bar{y}_0$. Because $f,g$ are odd functions of $y$, so is $\Phi_{C_g}$ as an odd function of $y_0$, i.e. $\Phi_{C_g}(t,\omega)(-y_0) = -\Phi_{C_g}(t,\omega)y_0$.\\
From \cite[Theorem 3.1]{duc21}, $\Phi_{C_g}$ admits a random pullback attractor in $\cA_{C_g}(\cdot) \subset \R$ that is a compact random set on $\R$ and that $|\cA_{C_g}(\cdot)|\in \cL^\rho$ for any $\rho \geq 1$. Hence there exist 
\[
c_-(C_g,\omega) := \min \{y: y \in \cA_{C_g}(\omega)\} \in \cL^\rho \quad c_+(C_g,\omega) := \max \{y: y \in \cA_{C_g}(\omega)\}\in \cL^\rho,\quad \forall \rho \geq 1.  
\]
The invariance of $\cA_{C_g}$ and the oder-preserving property of $\Phi_{C_g}$ imply that $c_-(C_g,\cdot), c_+(C_g,\cdot)$ are also invariant under $\Phi_{C_g}$. Due to the monotonicity, any solution starting from above $c_+(C_g,\omega)$ (respectively below $c_-(C_g,\omega)$) will be attracted into $c_+(C_g,\omega)$ (respectively $c_-(C_g,\omega)$) in the pullback sense. Because the trivial solution is also invariant and should also be attracted to $\cA_{C_g}$, it should belong to $\cA_{C_g}(\omega)$, implying $c_-(C_g,\omega) \leq 0 \leq c_+(C_g,\omega)$. The odd property of $\Phi_{C_g}$ thus yields $c_-(C_g,\omega) = -c_+(C_g,\omega)$. In particular, for any fixed $\epsilon>0$, there exists a $T= T(\epsilon,\omega)>0$ large enough such that
\begin{equation}\label{attractorex}
    \Phi_{C_g}(t,\theta_{-t}\omega) \sqrt{\alpha} \leq c_+(C_g,\omega) + \epsilon,\quad \forall t \geq T(\epsilon,\omega).
\end{equation}
Using similar arguments to the proof of \cite[Theorem 17]{ducjost25}, for a given $\omega$, we can choose $T(\epsilon,\omega)$ large enough and fix it, so that there exists an integrable random variable $\xi(T,\omega)$ such that
\begin{equation}\label{couplingex}
    \big|\Phi_{C_g}(T(\epsilon,\omega),\theta_{-T(\epsilon,\omega)}\omega)\sqrt{\alpha} - \Phi_0(T(\epsilon,\omega))\sqrt{\alpha}\big| \leq C_g \xi(T(\epsilon,\omega),\omega).
\end{equation}
It follows from \eqref{attractorex} and \eqref{couplingex} that
\[
\sqrt{\alpha} - C_g \xi(T(\epsilon,\omega),\omega)\leq \Phi_{C_g}(T(\epsilon,\omega),\theta_{-T(\epsilon,\omega)}\omega)\sqrt{\alpha} \leq C_+(C_g,\omega) +\epsilon.
\]
which implies 
\begin{equation}\label{cupest}
\sqrt{\alpha} - C_g \xi(T(\epsilon,\omega),\omega) -\epsilon \leq C_+(C_g,\omega).
\end{equation}
On the other hand, it follows from the upper semi-continuity of $\cA_{C_g}$ in $C_g$ that for $C_g = C_g(\epsilon,\omega)$ small enough,
\begin{equation}\label{cdownest}
C_+(C_g,\omega) \leq \sqrt{\alpha} +\epsilon.
\end{equation}
By taking $C_g \to 0$ in \eqref{cupest} and \eqref{cdownest}, we are able to prove that $\lim \limits_{C_g \to 0} c_+(C_g,\omega) = \sqrt{\alpha}$ almost surely. Since $c_+(C_g,\cdot)\in \cL^\rho$ is bounded from above (see \cite[Theorem 3.1]{duc21}), by Lebesgue dominated convergence theorem, $\lim \limits_{C_g \to 0} \E|c(C_g,\cdot)-\sqrt{\alpha} |^\rho =0$ for any $\rho \geq 1$. In particular,
\begin{eqnarray}\label{cest}
&& \lim \limits_{C_g \to 0} c_-(C_g,\cdot) = -\sqrt{\alpha};\quad \lim \limits_{C_g \to 0} c_+(C_g,\cdot) = \sqrt{\alpha}\quad \text{a.s. and} \notag\\
&& \lim \limits_{C_g \to 0} \E [c_-(C_g,\cdot)]^2 = \lim \limits_{C_g \to 0} \E [c_+(C_g,\cdot)]^2 = \alpha.    
\end{eqnarray}
Using \eqref{cest}, we are now able to apply \cite[Theorem 4.5]{duchongcong26} to show that for $C_g$ sufficiently small (independently of $\omega$)
\[
\E Df(c_{\pm}(C_g,\cdot)) =\alpha -3 \E c_\pm(C_g,\cdot)^2   \leq -2 \alpha +\epsilon <0.    
\]
As a result, the two stationary solutions $c_\pm(C_g,\omega)$ become locally exponentially stable provided that $C_g$ small enough (independent of $\omega$). 

We can now apply Theorem \ref{tamedattractor} to conclude that the tamed scheme
\[
y_{k+1} = y_k + \frac{(\alpha y_k -y_k^3)\Delta}{1+M(\alpha y_k -y_k^3)\Delta} + C_g \tanh(y_k) x_{k\Delta,(k+1)\Delta} + C_g^2 \tanh(y_k) \Big[1-\tanh^2(y_k)\Big] x^2_{k\Delta,(k+1)\Delta}
\]
generates a discrete RDS $\Phi^\Delta_{C_g}$, which admits a global random pullback attractor $\cA^\Delta_{C_g}(\omega) \subset \R$ of the form $\cA^\Delta_{C_g}(\omega) = [-c^\Delta(C_g,\omega),c^\Delta(C_g,\omega)]$ where $c^\Delta(C_g,\omega)$ is invariant under $\Phi^\Delta_{C_g}$. Using same arguments as above and taking into account Theorem \ref{tamedattractor}, we prove that
\begin{eqnarray*}
&& \lim \limits_{\Delta \to 0} c^\Delta(C_g,\cdot) = c_+(C_g,\cdot);\\
&& \lim \limits_{C_g \to 0} c^\Delta(C_g,\cdot) = c^\Delta; 
\end{eqnarray*}
in both the pathwise and the $L^1$ sense, where by Subsection \ref{deterattractor}, $\cA^\Delta_0 = [-c^\Delta,c^\Delta]$ is the global attractor of the deterministic tamed scheme.
\end{example}

		\section{Appendix}

        \subsection{Rough paths and the probabilistic setting}\label{roughpath}
		Let us briefly present the concept of rough paths in the simplest form, following   Friz  \& Hairer  \cite{frizhairer} and  Lyons \cite{lyons98}. For any finite dimensional vector space $W$, denote by $C([a,b],W)$ the space of all continuous paths $y: [a,b] \to W$ equipped with the sup norm $\|\cdot\|_{\infty,[a,b]}$ given by $\|y\|_{\infty,[a,b]}=\sup_{t\in [a,b]} \|y_t\|$, 
		where $\|\cdot\|$ is the norm in $W$. We write $y_{s,t}:= y_t-y_s$. For $p\geq 1$, denote by $C^{p{\rm-var}}([a,b],W)\subset C([a,b],W)$ the space of all continuous paths $y:[a,b] \to W$ of finite $p$-variation 
		\[
		\ltn y\rtn_{\tp,[a,b]} :=\left(\sup_{\cP([a,b])}\sum_{i=1}^n \|y_{t_i,t_{i+1}}\|^p\right)^{1/p} < \infty, 
		\]
		where the supremum is taken over the entire class of finite partitions $\cP[a,b]$ of $[a,b]$. 
		Also, for each $0<\alpha<1$, we denote by $C^{\alpha}([a,b],W)$ the space of H\"older continuous functions with exponent $\alpha$ on $[a,b]$ equipped with the norm
		\begin{equation}\label{holnorm}
			\|y\|_{\alpha,[a,b]}: = \|y_a\| + \ltn y\rtn_{\alpha,[a,b]},\quad \text{where} \quad \ltn y\rtn_{\alpha,[a,b]} :=\sup_{\substack{s,t\in [a,b],\ s<t}}\frac{\|y_{s,t}\|}{(t-s)^\alpha} < \infty.
		\end{equation}
		Let  $\alpha \in (\frac{1}{3},\frac{1}{2})$ and $x \in C^\alpha([a,b],\R^m)$. A couple $\bx=(x,\X) \in \R^m \oplus (\R^m \otimes \R^m)$, where
		\[
		\X \in C^{2\alpha}([a,b]^2,\R^m \otimes  \R^m):= \left\{\X \in C([a,b]^2,\R^m \otimes  \R^m):  \sup_{\substack{s, t \in [a,b],\ s<t}} \frac{\|\X_{s,t}\|}{|t-s|^{2\alpha}} < \infty \right\}, 
		\]
		is called a {\it rough path lift} if it satisfies Chen's relation
		\begin{equation}\label{chen}
			\X_{s,t} - \X_{s,u} - \X_{u,t} = x_{s,u} \otimes  x_{u,t},\qquad \forall a \leq s \leq u \leq t \leq b. 
		\end{equation}
        The two-parameter function $\X$ then postulates the values for the iterated integral
        \begin{equation}
            \X_{s,t}=\int_s^t x_{s,r}dx_r.
        \end{equation}
        Such integrals are needed for representing pathwise solutions of stochastic differential equations.
		We introduce the rough path semi-norm 
		\begin{equation}\label{translated}
			\ltn \bx \rtn_{\alpha,[a,b]} := \ltn x \rtn_{\alpha,[a,b]} + \ltn \X \rtn_{2\alpha,[a,b]^2}^{\frac{1}{2}},\quad \text{where}\quad \ltn \X \rtn_{2\alpha,[a,b]^2}:= \sup_{s, t \in [a,b];s<t} \frac{\|\X_{s,t}\|}{|t-s|^{2\alpha}} < \infty.  
		\end{equation}
		Throughout this paper, we will fix parameters $\frac{1}{3}< \alpha < \nu <\frac{1}{2}$ and  $p = \frac{1}{\alpha}$ so that $C^\alpha([a,b],W) \subset C^{\tp}([a,b],W)$. We also set $q=\frac{p}{2}$ and consider the $\tp$ semi-norm 
		\begin{equation}\label{pvarnorm}
			\begin{split}
				\ltn \bx \rtn_{\tp,[a,b]} := \Big(\ltn x \rtn^p_{\tp,[a,b]} + \ltn \X \rtn_{\tq,[a,b]^2}^q\Big)^{\frac{1}{p}}, \ltn \X \rtn_{\tq,[a,b]^2} := \left(\sup_{\cP([a,b])}\sum_{i=1}^n \|\X_{t_i,t_{i+1}}\|^q\right)^{1/q}, 
			\end{split}
		\end{equation}
		where the supremum is taken over the whole class of finite partitions $\cP([a,b])$ of $[a,b]$.  

        For the discrete time set $\Pi$, we introduce the notation of $\|y\|_{\infty,\Pi[a,b]},\ltn y\rtn_{\tp,\Pi[a,b]}$ and $\ltn \bx \rtn_{\tp,\Pi[a,b]}$ in the discrete time interval $\Pi[a,b]$ in a similar way to those in the continuous time interval $[a,b]$.

        Denote by $T^2_1(\R^m) = 1 \oplus \R^m \oplus (\R^m \otimes \R^m)$ the set with the group product
	\[
	(1,g^1,g^2) \bullet (1,h^1,h^2) = (1, g^1 + h^1, g^1 \otimes h^1 + g^2 +h^2),
	\]
	for all ${\bf g} =(1,g^1,g^2), {\bf h} = (1,h^1,h^2) \in T^2_1(\R^m)$. Denote by $\cC^{0,\alpha}(I,T^2_1(\R^m))$ 
	the closure of $\cC^{\infty}(I,T^2_1(\R^m))$ in the H\"older space $\cC^{\alpha}(I,T^2_1(\R^m))$, and by 
	$\cC_0^{0,\alpha}(\R,T^2_1(\R^m))$ the space of all paths $ {\bf g}: \R\to T^2_1(\R^m))$ such that $ {\bf g}|_I \in \cC^{0,\alpha}(I, T^2_1(\R^m))$ for each compact interval $I\subset\R$ containing $0$. Assign $\Omega:=\cC_0^{0,\alpha}(\R,T^2_1(\R^m))$ and equip it with the Borel $\sigma$-algebra $\mathcal{F}$. Denote by $\theta$ the {\it Wiener-type shift}
	\begin{equation}\label{shift}
		(\theta_t \omega)_\cdot = \omega_t^{-1}\bullet \omega_{t+\cdot},\forall t\in \R, \omega \in \cC^{0,\alpha}_0(\R,T^2_1(\R^m)),
	\end{equation}  
	and define the so-called {\it diagonal process}  $\bX: \R \times \Omega \to T^2_1(\R^m), \bX_t(\omega) = \omega_t$ for all $t\in \R, \omega \in \Omega$. Under assumption 	(${\textbf H}_X$), it can be proved that there exists a probability measure $\bP$ which is $\theta$ - invariant \cite[Theorem 5]{BRSch17}. Thus $(\Omega,\mathcal{F},\mP)$ is a probability space equipped with the continuous (thus measurable) {\it metric dynamical system} $\theta: \R \times \Omega \to \Omega$. 
	In particular, the Wiener shift \eqref{shift} implies that
	\begin{equation}\label{roughshift}
		\ltn \bx(\theta_h \omega) \rtn_{\tp,[s,t]} = \ltn \bx(\omega) \rtn_{\tp,[s+h,t+h]}.
	\end{equation}
	It is proved in \cite[Lemma 6.1]{duchongcong26} that $\theta$ is ergodic if $X=B^H$ is a fractional Brownian motion. In this paper, we assume that the metric dynamical system $\theta$ is ergodic. 

    Note that when dealing with additive noise, we do not need rough path lifts but instead consider $\Omega:=\cC_0^{0,\alpha}(\R,\R^m)$ together with a Wiener shift $(\theta_t \omega)_\cdot = \omega_{t+\cdot} -\omega_t$.
	\begin{lemma}\label{chenlem}
For any $n \geq 1$, any sequence $t_0 < t_1 < \ldots < t_n$ and any constant $C>0$,	the following estimate holds
\begin{equation}\label{chenest} 
1+ C \|x_{t_0,t_n}\| + C^2 \|\X_{t_0,t_n}\| \leq \prod_{i = 0}^{n-1} \big(1+ C \|x_{t_i,t_{i+1}}\| + C^2 \|\X_{t_i,t_{i+1}}\|\big).
\end{equation}	
\end{lemma} 	
\begin{proof}
The proof is obvious by induction. The case $n =1$ holds trivially, so it is sufficient to prove  \eqref{chenest} for $n=2$. Using Chen's relation \eqref{chen} and the triangle inequality, we obtain
\begin{eqnarray*}
&& \big(1+ C \|x_{t_0,t_1}\| + C^2 \|\X_{t_0,t_1}\|\big)\big(1+ C \|x_{t_1,t_2}\| + C^2 \|\X_{t_1,t_2}\|\big) \\
&\geq& 1 + C (\|x_{t_0,t_1}\| + \|x_{t_1,t_2}\|) + C^2 (\|\X_{t_0,t_1}\|+ \|\X_{t_1,t_2}\| + \|x_{t_0,t_1}\| \|x_{t_1,t_2}\| ) \\
&\geq& 1 + C\|x_{t_0,t_2}\| + C^2 \|\X_{t_0,t_2}\|
\end{eqnarray*}
which proves \eqref{chenest} holds for $n=2$. 
\end{proof} 	
     
\subsection{Stopping time analysis}\label{Sectime}
		\subsubsection*{Stopping times for the continuous time case}
		Following \cite{duchongcong26}, for any fixed $\gamma \in (0,1)$ any any closed interval $I\subset [0,\infty)$, we define another sequence 
        $\{\tau_i(\gamma,\bx,I)\}_{i \in \N}$  by
		\begin{equation}\label{greedytime*}
			\tau_0 = \min{I},\quad \tau_{i+1}:= \inf\Big\{t>\tau_i:  \ltn \bx \rtn_{\tp, [\tau_i,t]}  = \gamma \Big\} \wedge \max{I}.
		\end{equation}
		Define $N(\gamma,\bx,I) :=\sup \{i \in \N: \tau_i < \max{I}\}+1$. It is easy to show a rough estimate 
		\begin{equation}\label{Nest}
			N(\gamma,\bx ,I) \leq 1 +\frac{1}{\gamma^p}\ltn \bx \rtn^p_{\tp,I}.
		\end{equation}

		\subsubsection*{Stopping times for discrete time sets}\label{dis_greedy}
   For the given discrete time set $\Pi$, Let $\gamma>0$ be a parameter. Assign $\tau^\Pi_0 (\gamma,\bx,[0,\infty))= 0$. For each $n\in \N$, assume $\tau^\Pi_n(\gamma,\bx,[0,\infty)) =t_k$ is determined, one can define $\tau^\Pi_{n+1}(\gamma,\bx, [0,\infty))$ by the following rule:  \begin{itemize}
			\item if $\ltn \bx \rtn_{\tp,\Pi[t_k,t_{k+1}]} > \gamma$ then set $\tau^\Pi_{n+1}(\gamma,\bx,[0,\infty)) := t_{k+1}$;
			\item else set $\tau^\Pi_{n+1}(\gamma,\bx,[0,\infty)):= \sup \{t_l >t_k: \ltn \bx \rtn_{\tp,\Pi[t_k,t_l]}  \leq \gamma \}$.	
		\end{itemize}   
        Denote $N(\gamma,\bx,\Pi(I))$ to be the number of stopping times $\tau^\Pi_n$ on the discrete interval $\Pi(I)$ of $\Pi$. By definition, between two consecutive stopping times $\tau$ there are at most two stopping times $\tau^\Pi$, hence
        \begin{equation}\label{Nestdiscrete}
            N(\gamma,\bx,\Pi(I)) \leq 2 N(\gamma,\bx,I).
        \end{equation}

\subsection{Proofs}\label{proofs}

\begin{proof}[{\bf Lemma \ref{xpvarpoly}}]
 i,   The proof follows from Markov inequality and Borel-Cantelli lemma. Indeed, from (${\textbf H}_X$) it follows that $\E \ltn \bx(\omega) \rtn^\rho_{\tp,[-n-1,-n]} = \E\ltn \bx(\omega) \rtn^\rho_{\tp,[-1,0]} <\infty$ for all $n \in \N$. By Markov inequality
    \begin{eqnarray*}
    \sum_{n=1}^\infty \mP\Big(\ltn \bx(\omega) \rtn^\rho_{\tp,[-n-1,-n]} > n^{r_0}\Big)& \leq& \sum_{n=1}^\infty \frac{\E \ltn \bx(\omega) \rtn^\rho_{\tp,[-n-1,-n]}}{n^{r_0}} \\
    &= &\E \ltn \bx(\omega) \rtn^\rho_{\tp,[-1,0]}\sum_{n=1}^\infty \frac{1}{n^{r_0}} <\infty.
    \end{eqnarray*}
    Hence, it follows from Borel-Cantelli lemma that 
    \[
    \limsup \limits_{n \to \infty} \frac{\ltn \bx(\omega) \rtn^\rho_{\tp,[-n-1,-n]}}{n^{r_0}} \leq 1\quad \text{a.s.}; 
    \]
    That means for any $\kappa >0$ fixed, there exists a random variable $n(\kappa,\omega) \in \N$ such that for any $\omega$ in a full measure subset of $\Omega$, 
    \[
    \ltn \bx(\omega) \rtn^\rho_{\tp,[-n-1,-n]} \leq (1+\kappa) n^{r_0},\quad \forall n \geq n(\kappa,\omega).
    \]
    By choosing $C(\omega):= (1+\kappa)\max \Big\{\ltn \bx(\omega) \rtn^\rho_{\tp,[-k-1,-k]}: 0\leq k \leq n(\kappa,\omega)\Big\}$, we obtain \eqref{xpvarpoly1}.\\

    ii, If $X$ is Gaussian, it follows from \cite{frizhauser} that $\E e^{\ltn \bx \rtn_{\tp,[0,1]}} <\infty$. As a result, by Markov inequality and Stirling's formula, for sufficiently large $n$
    \begin{eqnarray*}
             \Big(\E \ltn \bx(\omega) \rtn_{\tp,[-1,0]}^{\rho n}\Big)^{\frac{1}{n}} &\leq& \Big(\int_0^\infty u^{\rho n} \mP(\ltn \bx(\omega) \rtn_{\tp,[-1,0]} > u)du\Big)^{\frac{1}{n}}\\
             &\leq& \Big(\int_0^\infty u^{\rho n} \mP(e^{\ltn \bx(\omega) \rtn_{\tp,[-1,0]}} > e^u)du\Big)^{\frac{1}{n}}\\
&\leq& \Big(\E e^{\ltn \bx(\omega) \rtn_{\tp,[-1,0]}}\Big)^{\frac{1}{n}} \Big(\int_0^\infty u^{\rho n} e^{-u} du\Big)^{\frac{1}{n}}\\
 &\leq& \Big(\E e^{\ltn \bx(\omega) \rtn_{\tp,[-1,0]}}\Big)^{\frac{1}{n}} \Big(\Gamma(\rho n +1)\Big)^{\frac{1}{n}}\\
 &\leq& C_0 \Big(\sqrt{2\pi \rho n} \big(\frac{\rho n}{e}\big)^{\rho n} \Big)^{\frac{1}{n}}\\
 &\leq& C_0 \Big[1+ \Big(\frac{\rho}{e}\Big)^\rho\Big](1+n^\rho).
    \end{eqnarray*}
 We choose a generic constant $C_0$ large enough such that the above inequality holds for all $n\in \N$, which proves \eqref{Ebx} by choosing $r_0:= \rho$.   \\
\end{proof}

\begin{proof}[{\bf Lemma \ref{lemmaF1}}]
    First, assign 
\[
S_0(\omega):= \xi(\omega),\quad S_k (\omega) := \sum_{i = 1}^k \xi(\theta_i \omega),\quad S_{-k} (\omega) := \sum_{i = 1}^k \xi(\theta_{-i} \omega),\quad \forall k\in \N, k \geq 1.
\]
Then by Birkhorff ergodic theorem $\lim \limits_{n\to \infty}\frac{1}{n}S_{-n}(\omega) = \E\xi>0$. Hence for a.s. each $\omega \in \Omega$ fixed and any given $\kappa >0$, there exists $m(\omega,\kappa)$ such that, simultaneously 
\begin{equation}\label{momegakappa}
\begin{split}
& \frac{1}{k}S_{-k}(\omega) < \E\xi+\kappa,\quad \forall k \geq m(\omega,\kappa);\\
&1+\epsilon d(\omega) < m(\omega,\kappa) \kappa.
\end{split}
\end{equation}
To estimate the sum inside the exponential function in \eqref{Rlem}, observe from \eqref{momegakappa} that
\begin{eqnarray*}
    &&\sum_{k=0}^{n-1} \frac{-\delta}{1+ \epsilon d(\omega) + \epsilon S_{-k}(\omega)} \\
    &=&     \sum_{k=0}^{m(\omega,\kappa)} \frac{-\delta}{1+ \epsilon d(\omega) + \epsilon S_{-k}(\omega)} +     \sum_{k=m(\omega,\kappa)+1}^{n-1} \frac{-\delta}{1+ \epsilon d(\omega) + \epsilon S_{-k}(\omega)}\\
    &\leq&     \sum_{k=0}^{m(\omega,\kappa)} \frac{-\delta}{1+ \epsilon d(\omega) + \epsilon S_{-k}(\omega)} +     \sum_{k=m(\omega,\kappa)+1}^{n-1} \frac{-\delta}{1+ \epsilon d(\omega) + k \epsilon (\E \xi+\kappa)}\\
    &\leq&     \sum_{k=0}^{m(\omega,\kappa)} \frac{-\delta}{1+ \epsilon d(\omega) + \epsilon S_{-k}(\omega)}     -\frac{\delta}{\epsilon (\E \xi+2\kappa)} \Big[\sum_{k=1}^{n-1}\frac{1}{k} - \sum_{j=1}^{m(\omega,\kappa)}\frac{1}{k}\Big]\\
    &\leq&     \sum_{k=0}^{m(\omega,\kappa)} \frac{-\delta}{1+ \epsilon d(\omega)+ \epsilon S_{-k}(\omega)} -\frac{\delta}{\epsilon (\E \xi+2\kappa)} \Big(\log n - \Big[\log m(\omega,\kappa)+1\Big]\Big),
\end{eqnarray*}
where we use the fact that
\begin{eqnarray}\label{zetaf}
0&<&\Big(\sum_{k=1}^{n-1} \frac{1}{k}-\log n \Big)- \Big(\sum_{k=1}^{m(\omega,\kappa)} \frac{1}{k}-\log [m(\omega,\kappa)+1]\Big)\\
&=&  \sum_{k=m(\omega,\kappa)+1}^{n-1} \Big(\frac{1}{k}-\int_{k}^{k+1}\frac{1}{u}du\Big) = \sum_{k=m(\omega,\kappa)+1}^{n-1} \int_{0}^{1}\frac{udu}{k(u+k)}  <  \sum_{k=m(\omega,\kappa)+1}^\infty \frac{\int_0^1 udu}{k^2}< \frac{1}{2}\zeta(2)    \notag
\end{eqnarray}
where $\zeta(\cdot)$ is the zeta function. Assign
\[
\rho (\kappa,\epsilon):= \frac{\delta}{\epsilon (\E \xi+2\kappa)} >0.
\]
Because of \eqref{deltaeps}, we can choose $\kappa$ small enough so that 
 \[
 \rho(\kappa,\epsilon) - r(\xi)= \frac{\delta}{\epsilon (\E \xi+2\kappa)}-r(\xi)>1. 
 \]
As a result, 
\begin{eqnarray}\label{Rest}
R(\omega)
    &=& \sum_{n=0}^{m(\omega,\kappa)} \xi(\theta_{-n} \omega) \exp \Big\{\sum_{j =0}^{n-1}\frac{-\delta}{1+ \epsilon d(\omega)+ \epsilon S_{-j}(\omega)}\Big\}\notag\\
    &&+\sum_{n=m(\omega,\kappa)+1}^{\infty} \xi(\theta_{-n} \omega) \exp \Big\{\sum_{j =0}^{n-1}\frac{-\delta}{1+ \epsilon d(\omega) + \epsilon S_{-j}(\omega)}\Big\} \notag\\
    &=& \sum_{n=0}^{m(\omega,\kappa)} \xi(\theta_{-n} \omega) \exp \Big\{\sum_{j =0}^{n-1}\frac{-\delta}{1+ \epsilon d(\omega) + \epsilon S_{-j}(\omega)}\Big\}\notag\\
    &&+\exp \Big\{\sum_{j =0}^{m(\omega,\kappa)}\frac{-\delta}{1+ \epsilon d(\omega) + \epsilon S_{-j}(\omega)}\Big\}\times \notag\\
    &&\times \sum_{n=m(\omega,\kappa)+1}^{\infty} \xi(\theta_{-n} \omega) \exp \Big\{\sum_{j =m(\omega,\kappa)+1}^{n-1}\frac{-\delta}{1+ \epsilon d(\omega) + \epsilon S_{-j}(\omega)}\Big\}. 
\end{eqnarray}
The last term in \eqref{Rest} can be estimated, due to \eqref{zetaf} and  \eqref{xipoly}, as
\begin{eqnarray}\label{seriesfinite}
&&\sum_{n=m(\omega,\kappa)+1}^{\infty} \xi(\theta_{-n} \omega) \exp \Big\{\sum_{j =m(\omega,\kappa)+1}^{n-1}\frac{-\delta}{1+ \epsilon d(\omega) + \epsilon S_{-j}(\omega)}\Big\}\notag\\
&\leq&  \sum_{n=m(\omega,\kappa)+1}^{\infty} \xi(\theta_{-n} \omega) \exp \Big\{   -\frac{\delta}{\epsilon (\E \xi+2\kappa)} \Big(\log n - \log \Big[m(\omega,\kappa)+1\Big]\Big)
\Big\} \notag\\
&\leq& \Big[m(\omega,\kappa)+1\Big]^{\rho(\kappa,\epsilon)} C(\xi,\omega)\sum_{n=m(\omega,\kappa)+1}^{\infty} \frac{1+n^{r(\xi)} }{n^{\rho(\kappa,\epsilon)}}\notag\\
&\leq& \Big[m(\omega,\kappa)+1\Big]^{\rho(\kappa,\epsilon)} C(\xi,\omega) \Big[\zeta\big(\rho(\kappa,\epsilon)\big) + \zeta\big(\rho(\kappa,\epsilon) -r(\xi)\big) \Big] <\infty.
 \end{eqnarray}
 Hence, $R(\omega) < \infty$ almost sure.  \\

 To prove the integrability of $R$, observe that for $\delta <2$, the following function is concave on $\R_+$
    \begin{eqnarray*}
        h(u) = \exp \Big\{-\frac{\delta }{1+\epsilon u}\Big\} \Rightarrow \frac{d^2}{du^2}h(u) = h(u) \frac{\delta \epsilon^2}{(1+\epsilon u)^4} (\delta -2-2\epsilon u) <0.
    \end{eqnarray*} 
 We then apply H\"older inequality for the expectation and Jensen inequality for the concave function $h$, to obtain, due to \eqref{Exi},
 \begin{eqnarray*}
     \E R(\cdot) &=& \sum_{n=0}^{\infty} \E\xi(\theta_{-n} \omega) \exp \Big\{\sum_{j =0}^{n-1}\frac{-\delta}{1+ \epsilon d(\omega) + \epsilon S_{-j}(\omega)}\Big\} \notag\\
     &\leq& \sum_{n=0}^{\infty} \Big(\E\xi(\theta_{-n} \omega)^{n+1}\Big)^{\frac{1}{n+1}} \prod_{j =0}^{n-1} \Big(\E\exp \Big\{\frac{-(n+1)\delta}{1+ \epsilon d(\omega) + \epsilon S_{-j}(\omega)}\Big\}\Big)^{\frac{1}{n+1}}\\
     &\leq& \sum_{n=0}^{\infty} \Big(\E\xi^{n+1}\Big)^{\frac{1}{n+1}} \prod_{j =0}^{n-1} \Big(\exp \Big\{\frac{-(n+1)\delta}{1+ \epsilon \E d(\cdot) + \epsilon \E S_{-j}(\cdot)}\Big\}\Big)^{\frac{1}{n+1}}\\
     &\leq& \sum_{n=0}^{\infty} C_0(\xi) \Big[1+(n+1)^{r(\xi)}\Big]  \exp \Big\{\sum_{j =0}^{n-1}\frac{-\delta}{1+ \epsilon \E d(\cdot) + \epsilon j \E \xi(\cdot)}\Big\}.
 \end{eqnarray*}
 Using similar estimates to \eqref{Rest} and \eqref{seriesfinite}, we can show that the last series is finite, hence $R \in L^1$.\\
\end{proof}

\begin{proof}[{\bf Lemma \ref{lemFprop}}]
It is obvious that $F$ is increasing in $\epsilon$. A direct computation shows that
    \begin{eqnarray*}
      0<  \frac{\partial F}{\partial u} (\xi,u) &=& \exp \Big\{-\frac{\delta }{1+\epsilon u+ \epsilon \xi}\Big\} \Big[1+\frac{\delta \epsilon u}{(1+\epsilon \xi +\epsilon u)^2} \Big]\\
      &\leq& \frac{1+\frac{\delta \epsilon u}{(1+\epsilon \xi +\epsilon u)^2}}{1+\frac{\delta }{1+\epsilon u+ \epsilon \xi} +\frac{\delta^2 }{(1+\epsilon u+ \epsilon \xi)^2}} \\
      &=& \frac{(1+\epsilon u+ \epsilon \xi)^2+\delta \epsilon u}{(1+\epsilon u+ \epsilon \xi)^2 + \delta (1+\epsilon u+ \epsilon \xi) +\delta^2}\\
      &=& 1- \frac{\delta(1+\delta+\epsilon \xi)}{(1+\epsilon u+ \epsilon \xi)^2 + \delta (1+\epsilon u+ \epsilon \xi) +\delta^2}\\
      &\leq& 1- \frac{\delta(1+\delta+\epsilon \xi)}{(1+\delta+\epsilon \xi +\epsilon u)^2}
    \end{eqnarray*}
    which proves \eqref{Fprop}.\\
\end{proof}

\begin{proof} [{\bf Lemma \ref{lemmaF2}}]
It follows from \eqref{uk1F} and \eqref{F2} that $u_{k+1}(\omega) < u_k(\omega) + \xi(\theta_{k+1}\omega)$, hence by induction
\[
u_k(\omega) \leq u_0(\omega) +S_k(\omega) ,\quad  \forall 1\leq k \in \N.
\]
Since the function $\exp \big\{-\frac{\delta}{1+ \epsilon u + \epsilon \xi}\big\}$ is increasing in $u$, it follows that 
\begin{eqnarray*}
    u_{k+1}(\omega) &\leq& u_k(\omega) \exp \Big\{-\frac{\delta}{1+ \epsilon \big(u_0(\omega) + S_k(\omega)+ \xi(\theta_{k+1}\omega) \big)}\Big\} +\xi(\theta_{k+1}\omega)\notag\\
    &\leq& u_k \exp \Big\{-\frac{\delta}{1+ \epsilon u_0(\omega) + \epsilon S_{k+1}(\omega)}\Big\}+\xi(\theta_{k+1}\omega),\quad \forall k\geq 1.
\end{eqnarray*}
Therefore, it is easy to prove by induction that
\begin{eqnarray*}\label{uF2}
    u_n(\omega) &\leq& u_0(\omega) \prod_{k =1}^n \exp \Big\{-\frac{\delta}{1+ \epsilon u_0(\omega) + \epsilon S_k(\omega)}\Big\}+ \sum_{k=1}^n \xi(\theta_k \omega)\prod_{j =k+1}^n \exp \Big\{-\frac{\delta}{1+ \epsilon u_0(\omega) + \epsilon S_j(\omega)}\Big\}\notag\\
    &\leq& u_0(\omega) \exp\Bigg\{ \sum_{k=1}^n \frac{-\delta}{1+ \epsilon u_0(\omega) + \epsilon S_k(\omega)}\Bigg\} + \sum_{k=1}^n \xi(\theta_k \omega) \exp \Big\{\sum_{j =k+1}^n\frac{-\delta}{1+ \epsilon u_0(\omega) + \epsilon S_j(\omega)}\Big\}.
\end{eqnarray*}
Replacing $\omega$ by $\theta_{-n}(\omega)$ yields
\begin{eqnarray}\label{uF3}
u_n(\theta_{-n}\omega) &\leq& u_0(\theta_{-n}\omega) \exp\Bigg\{ \sum_{k=1}^n \frac{-\delta}{1+ \epsilon u_0(\theta_{-n}\omega) + \epsilon S_k(\theta_{-n}\omega)}\Bigg\}\notag\\
&&+ \sum_{k=1}^n \xi(\theta_{k-n} \omega) \exp \Big\{\sum_{j =k+1}^n\frac{-\delta}{1+ \epsilon u_0(\theta_{-n}\omega) + \epsilon S_j(\theta_{-n}\omega)}\Big\} \notag\\
&\leq& u_0(\theta_{-n}\omega) \exp\Bigg\{ \sum_{k=1}^n \frac{-\delta}{1+ \epsilon u_0(\theta_{-n}\omega) + \epsilon \sum_{i=n-k}^{n-1} \xi(\theta_{-i}\omega)}\Bigg\}\notag\\
&&+ \sum_{k=1}^n \xi(\theta_{k-n} \omega) \exp \Big\{\sum_{j =k+1}^n\frac{-\delta}{1+ \epsilon u_0(\theta_{-n}\omega) + \epsilon \sum_{i=n-j}^{n-1} \xi(\theta_{-i}\omega)}\Big\} \notag\\
&\leq& u_0(\theta_{-n}\omega) \exp\Bigg\{ \sum_{k=0}^{n-1} \frac{-\delta}{1+ \epsilon u_0(\theta_{-n}\omega) + \epsilon \sum_{i=1}^{k} \xi(\theta_{-i}\omega)}\Bigg\}\notag\\
&&+ \sum_{k=1}^n \xi(\theta_{k-n} \omega) \exp \Big\{\sum_{j =0}^{n-k-1}\frac{-\delta}{1+ \epsilon u_0(\theta_{-n}\omega) + \epsilon \sum_{i=1}^{j} \xi(\theta_{-i}\omega)}\Big\} \notag\\
&\leq& u_0(\theta_{-n}\omega) \exp\Bigg\{ \sum_{k=0}^{n-1} \frac{-\delta}{1+ \epsilon u_0(\theta_{-n}\omega) + \epsilon \sum_{i=1}^{k} \xi(\theta_{-i}\omega)}\Bigg\}\notag\\
&&+ \sum_{k=0}^{n-1} \xi(\theta_{-k} \omega) \exp \Big\{\sum_{j =0}^{k-1}\frac{-\delta}{1+ \epsilon u_0(\theta_{-n}\omega) + \epsilon \sum_{i=1}^{j} \xi(\theta_{-i}\omega)}\Big\}\notag\\
&\leq& \|u_0(\omega)\|_\infty \exp\Bigg\{\sum_{k=0}^{n-1} \frac{-\delta}{1+ \epsilon \|u_0(\omega)\|_\infty + \epsilon S_{-k}(\omega)}\Bigg\}\notag\\
&&+ \sum_{k=0}^{n-1} \xi(\theta_{-k} \omega) \exp \Big\{\sum_{j =0}^{k-1}\frac{-\delta}{1+ \epsilon \|u_0(\omega)\|_\infty + \epsilon S_{-j}(\omega)}\Big\}
\end{eqnarray}
which proves \eqref{uk2F}.\\

In addition, for any two sequences $\{u_k(\omega,u_0(\omega))\}_{k\in \N}, \{\bar{u}_k(\omega,\bar{u}_0(\omega))\}_{k\in \N}$, it follows from the Lagrange mean value theorem and \eqref{Fprop} that there exists $\kappa \in (0,1)$ such that
\begin{eqnarray*}
    |u_{k+1}(\omega)-\bar{u}_{k+1}(\omega)| &=& |F(\xi(\theta_{k+1}\omega),u_k) -F(\xi(\theta_{k+1}\omega),\bar{u}_k)|\\
    &=& \Big| \frac{\partial F}{\partial u} (\xi(\theta_{k+1}\omega),\kappa u_k + (1-\kappa)\bar{u}_k)\Big| |u_k(\omega)-\bar{u}_k(\omega)|\\
    &\leq& \Big(1- \frac{\delta(1+\delta+\epsilon \xi(\theta_{k+1}\omega))}{\Big[1+\delta+\epsilon \xi(\theta_{k+1}\omega) +\epsilon \big(\kappa u_k(\omega) + (1-\kappa)\bar{u}_k(\omega)\big)\Big]^2}\Big) |u_k(\omega)-\bar{u}_k(\omega)|\\
    &\leq& \exp \Big\{\frac{-\delta(1+\delta+\epsilon \xi(\theta_{k+1}\omega))}{\Big[1+\delta+\epsilon \xi(\theta_{k+1}\omega) +\epsilon \big( u_k(\omega) \vee \bar{u}_k(\omega)\big)\Big]^2} \Big\} |u_k(\omega)-\bar{u}_k(\omega)|.
\end{eqnarray*}
Hence by induction,
\begin{equation}\label{udiff}
    |u_n(\omega)-\bar{u}_n(\omega)| \leq \prod_{k=0}^{n-1} \exp \Big\{\frac{-\delta(1+\delta+\epsilon \xi(\theta_{k+1}\omega))}{\Big[1+\delta+\epsilon \xi(\theta_{k+1}\omega) +\epsilon \big( u_k(\omega) \vee \bar{u}_k(\omega)\big)\Big]^2} \Big\} |u_0(\omega)-\bar{u}_0(\omega)|.  
\end{equation}
By replacing $\omega$ by $\theta_{-n}\omega$ with \eqref{udiff} and using \eqref{uk2F}, we obtain
\begin{eqnarray}\label{udiff2}
    &&|u_n(\theta_{-n}\omega)-\bar{u}_n(\theta_{-n}\omega)| \notag\\
    &\leq& \prod_{k=0}^{n-1} \exp \Big\{\frac{-\delta(1+\delta+\epsilon \xi(\theta_{k+1}\circ \theta_{-n}\omega))}{\Big[1+\delta+\epsilon \xi(\theta_{k+1}\circ \theta_{-n}\omega) +\epsilon \big( u_k(\theta_{-n}\omega) \vee \bar{u}_k(\theta_{-n}\omega)\big)\Big]^2} \Big\} |u_0(\theta_{-n}\omega)-\bar{u}_0(\theta_{-n}\omega)|\notag\\
    &\leq& \prod_{k=0}^{n-1} \exp \Big\{\frac{-\delta(1+\delta+\epsilon \xi(\theta_{k+1-n}\omega))}{\Big[1+\delta+\epsilon \xi(\theta_{k+1-n}\omega) +\epsilon \big( u_k(\theta_{-k}\circ  \theta_{k-n}\omega) \vee \bar{u}_k(\theta_{-k}\circ \theta_{k-n}\omega)\big)\Big]^2} \Big\} \notag\\
    &&\times|u_0(\theta_{-n}\omega)-\bar{u}_0(\theta_{-n}\omega)|\notag\\
    &\leq& \exp \Bigg\{\sum_{k=0}^{n-1} \frac{-\delta(1+\delta+\epsilon \xi(\theta_{k+1-n}\omega))}{\Big[1+\delta+\epsilon \xi(\theta_{k+1-n}\omega) +\epsilon \bar{d}(\theta_{k-n}\omega) \Big]^2}  \Bigg\} \Big(\|u_0(\omega)\|_\infty+\|\bar{u}_0(\omega)\|_\infty\Big) 
\end{eqnarray}
where
\[
\bar{d}(\omega) =\|u_0(\omega)\|_\infty \vee \|\bar{u}_0(\omega)\|_\infty  + R\Big(\epsilon,\|u_0(\omega)\|_\infty \vee \|\bar{u}_0(\omega)\|_\infty ,\xi(\omega)\Big) \in L^1.
\]
Taking the logarithm in both sides of \eqref{udiff2} and then dividing by $n$ and letting $n$ to infinity and applying Birkhorff ergodic theorem, we obtain
\begin{eqnarray}
   &&\limsup \limits_{n \to \infty} \frac{1}{n} \log   |u_n(\theta_{-n}\omega)-\bar{u}_n(\theta_{-n}\omega)| \notag\\
   &\leq&  \limsup \limits_{n \to \infty} \frac{1}{n} \log \Big(\|u_0(\omega)\|_\infty+\|\bar{u}_0(\omega)\|_\infty\Big) +  \limsup \limits_{n \to \infty} \frac{1}{n} \sum_{k=0}^{n-1} \frac{-\delta(1+\delta+\epsilon \xi(\theta_{k+1-n}\omega))}{\Big[1+\delta+\epsilon \xi(\theta_{k+1-n}\omega) +\epsilon \bar{d}(\theta_{k-n}\omega) \Big]^2}\notag\\
   &\leq& \limsup \limits_{n \to \infty} \frac{1}{n} \sum_{k=1}^{n} \frac{-\delta(1+\delta+\epsilon \xi(\theta_{-k+1}\omega))}{\Big[1+\delta+\epsilon \xi(\theta_{-k+1}\omega) +\epsilon \bar{d}(\theta_{-k}\omega) \Big]^2} = -\E \frac{\delta(1+\delta+\epsilon \xi(\theta_{1}\cdot))}{\Big[1+\delta+\epsilon \xi(\theta_{1}\cdot) +\epsilon \bar{d}(\cdot) \Big]^2} <0
\end{eqnarray}
where we use the fact that $\frac{\delta(1+\delta+\epsilon \xi(\theta_{1}\cdot))}{[1+\delta+\epsilon \xi(\theta_{1}\cdot) +\epsilon \bar{d}(\cdot) ]^2} \in (0,1)$. In other words, 
\begin{equation}\label{udiff1}
\limsup \limits_{n \to \infty}  |u_n(\theta_{-n}\omega)-\bar{u}_n(\theta_{-n}\omega)| = 0\quad \text{exponentially and a.s.}
\end{equation}
As a result, for any sequence $\{u_k(\omega)\}$ satisfying \eqref{u0}, the omega limit set $u_\infty(\omega)$ defined by \eqref{pullbacku} is bounded from above due to \eqref{uk2F} and closed by definition, hence \eqref{uinfty} follows from \eqref{uF3} by taking $n\to \infty$. For any point $r(\omega) \in u_\infty(\omega)$, there exists a subsequence $n_i \to \infty$ such that $u_{n_i}(\theta_{-n_i}\omega) \to r(\omega)$ as $n_i \to \infty$. Together with \eqref{udiff1}, it follows that $\bar{u}_{n_i}(\theta_{-n_i}\omega) \to r(\omega)$ as $n_i \to \infty$, which implies that $r(\omega) \in \bar{u}_\infty(\omega)$ for any other sequence $\{\bar{u}_k(\omega)\}$ satisfying \eqref{u0}. This proves the independence of $u_\infty$ on $u_0$. The estimate \eqref{uinfty} is obvious from the definition of $u_\infty$. Finally, under condition \eqref{Exi}, the right hand side of \eqref{uk2F} is integrable, thus $|u_\infty(\cdot)| \in L^1$.\\
\end{proof}
		\bibliographystyle{apalike}

	\end{document}